\documentclass[a4paper,11pt,reqno,pdftex]{amsart}

\usepackage{amsmath,amssymb,amsfonts,amsthm}
\usepackage{ifthen}
\usepackage{mathtools}
\mathtoolsset{showonlyrefs,showmanualtags}
\usepackage{latexsym}
\usepackage{mathrsfs}
\usepackage[inline]{enumitem}
\newlist{assumplist}{enumerate}{1}
\setlist[assumplist]{label=(A\arabic*), ref=A\arabic*}
\newlist{Kolmogorovconditions}{enumerate}{1}
\setlist[Kolmogorovconditions]{label=(K\arabic*), ref=K\arabic*}
\newlist{ChowTeicher}{enumerate}{1}
\setlist[ChowTeicher]{label=(CT), ref=CT}
\newcommand{\descitem}[1]{%
  \refstepcounter{enumi}
  \item[(\roman{enumi})\ #1]
}
\usepackage{tikz,xcolor}
\usepackage{comment}
\numberwithin{equation}{section}

\theoremstyle{definition}
\newtheorem{theorem}{Theorem}[section]

\newtheorem{lemma}[theorem]{Lemma}
\newtheorem{corollary}[theorem]{Corollary}
\newtheorem{claim}[theorem]{Claim}
\newtheorem{remark}[theorem]{Remark}

\newcommand{\dlimit}{%
  \stackrel{\mathrm{d}}{\to}%
}
\newcommand{\floor}[1]{%
  \left\lfloor#1 \right\rfloor%
}
\newcommand{\ceil}[1]{%
  \left\lceil#1 \right\rceil%
}
\newcommand{\onetorus}{%
  \mathbb{T}^{1}%
}%
\newcommand{\Borelsigma}{%
  \mathcal{B}%
}%
\newcommand{\localmodulusofcont}{%
  \sigma_{l}%
}%
\newcommand{\seqspace}{%
  \mathbb{R}^{\mathbb{N}}%
}%

\title[ASIP for Takagi--van der Waerden class functions]{An almost sure invariance principle for the Takagi--van der Waerden class functions}
\author[Y. Nakano]{Yuzaburo Nakano}
\address{Graduate School of Engineering Science, Yokohama National University, Yokohama, Japan}
\email{nakano-yuzaburo-zg@ynu.jp}
\keywords{Nowhere differentiable functions; Probabilistic method; Elephant random walks; Almost sure invariance principles.}
\date{}

\begin{document}

\begin{abstract}
  The Takagi--van der Waerden functions are a well-known class of continuous but nowhere differentiable functions. In this paper, we study their weighted versions, the Takagi--van der Waerden class functions $f_{r,a}(x)$, from a probabilistic point of view. We prove that the local modulus of continuity of $f_{r,a}(x)$ is described by a standard Brownian motion under some regularity assumptions on the weights, as an application of a strong approximation for elephant random walks remembering the very recent past with variable step length.
\end{abstract}

\maketitle
\section{Introduction}


The \emph{Takagi--van der Waerden class functions} are defined by
\begin{equation}\label{eq:TvdWClass}
  f_{r,a}(x) := \sum_{k=1}^{\infty} \frac{a_k}{r^{k-1}}d(r^{k-1}x)
\end{equation}
for each integer $r\ge 2$, where $d(x)$ denotes the distance from $x$ to its nearest integer, and $\{a_k\colon k\ge 1\}$ is a deterministic real sequence satisfying
\begin{equation}\label{eq:conti_condition}
  \sum_{k=1}^{\infty} \frac{|a_k|}{r^{k-1}} < +\infty.
\end{equation}
If \eqref{eq:conti_condition} holds, then the series \eqref{eq:TvdWClass} converges uniformly, and are continuous. Remarkably, \eqref{eq:conti_condition} is also a necessary condition for the functions \eqref{eq:TvdWClass} to be continuous \cite[Theorem 1.2]{FerreraGomezGil20} (see also \cite{HataYamaguti84}).
When $a_k\equiv 1$, $f_{r,a}$ are a well-known class of continuous nowhere differentiable functions called the \emph{Takagi--van der Waerden functions}, say $f_r$ (see Fig.~\ref{fig:TakagivdWf2f3} for graphs of $f_2$ and $f_3$). The function $f_2$ was introduced by Takagi \cite{Takagi03}, and $f_{10}$ was independently studied by van der Waerden \cite{vanderWaerden30}. In \cite{NakanoTakeiTakagifunctions}, Nakano and Takei studied the modulus of continuity of $f_r$, and obtained a complete description of the differentiability properties of $f_{r,a}$ as applications of limit theorems for elephant random walks remembering the very recent past (ERWVRP). As the present paper is a continuation of the work \cite{NakanoTakeiTakagifunctions}, we refer the reader to its Introduction for detailed explanations on these topics.

\begin{figure}[t]
\label{fig:TakagivdWf2f3}
\begin{center}
\begin{tabular}{cc}
\includegraphics[width=5cm]{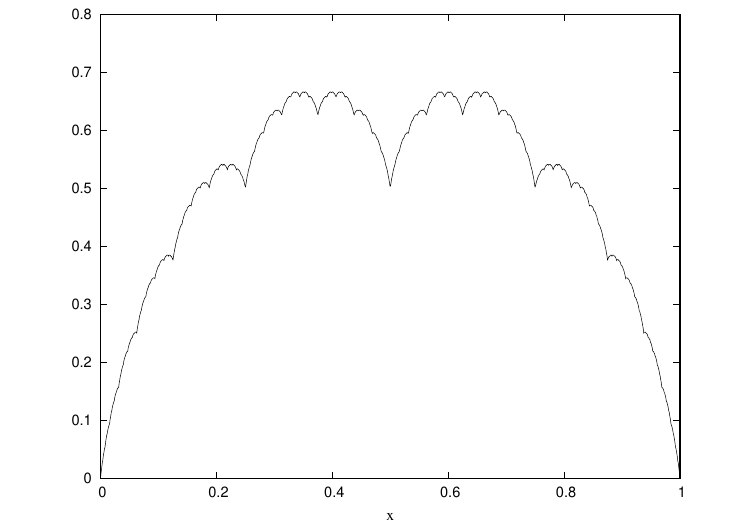}
&
\includegraphics[width=5cm]{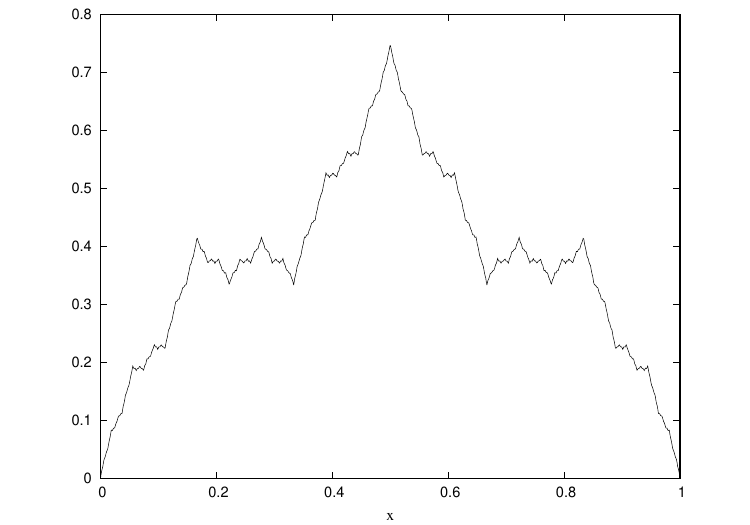}
\end{tabular}
\end{center}
\caption{The Takagi--van der Waerden functions $f_2(x)$ (left) and $f_3(x)$ (right).}
\end{figure}

The almost sure invariance principle (ASIP) is a powerful method in probability theory to prove limit theorems, such as the functional central limit theorem (FCLT), the law of the iterated logarithm (LIL), and upper and lower class results, among others. The ASIP states that one can construct an almost sure coupling between partial sums of random variables and a standard Brownian motion $\{B(t)\colon t\ge 0\}$ with sufficiently small error terms.
The ASIP was first introduced by Strassen \cite{Strassen64,Strassen67}. In \cite{Strassen64}, he proved that, for a sequence of independent identically distributed (i.i.d.) random variables $\{X_n\}_{n\ge 1}$ having mean $0$ and variance $1$, almost surely (a.s.),
\begin{equation}\label{eq:iid_ASIP_nloglogn}
  \sum_{k=1}^{n} X_k = B(n) + o(\sqrt{n\log\log n})\quad \text{as $n\to\infty$,}
\end{equation}
where, for real sequences $\{b_n\}_{n\ge 1}$ and $\{c_n\}_{n\ge 1}$, $b_n=o(c_n)$ as $n\to\infty$ means that $\lim_{n\to \infty} b_n/c_n =0$. The error term in \eqref{eq:iid_ASIP_nloglogn} is small enough to prove the Hartman--Wintner LIL \cite{HartmanWintner41}. After that, he extended his results to martingales with certain second moment conditions in \cite{Strassen67}. His proof relied on the Skorokhod embedding theorem (see Lemma \ref{lemma:Skorokhodembeddingtheorem} in Appendix \ref{appendix:some_technical_results}).

For dependent random variables, Philipp and Stout \cite{PhilippStout75ASIP} proved ASIPs for sums of weakly dependent random variables using blocking methods and the Skorokhod embedding theorem. For instance, they considered a trigonometric series defined on the Lebesgue probability space $(I , \Borelsigma(I), \mu)$ with the Hadamard gap of the form
\begin{equation}\label{eq:def_lacunary_series}
  f_n(\omega) := \sum_{k=1}^{n} a_k \cos(2\pi n_k \omega),
\end{equation}
where $n_{k+1}/n_{k} >q>1$, $I=[0,1)$, and $\Borelsigma(X)$ is the Borel $\sigma$-field of a set $X$.
They proved that there exists a richer probability space on which a standard Brownian motion $\{B(t)\colon t\ge 0\}$ and the sequence $\{\widetilde{f}_n\}_{n\ge 1}$ of random variables are defined such that $\{\widetilde{f}_n\}_{n\ge 1}$ has the same distribution as $\{f_n\}_{n\ge 1}$, and
\begin{equation}\label{eq:trigonometric_ASIP}
  \widetilde{f}_n = \frac{1}{\sqrt{2}}B(A_n) + o(A_n^{\frac{1}{2}-\lambda}) \quad \text{a.s.}
\end{equation}
for each $\lambda<\delta/32$ under the assumptions (see \cite[Chapter 6]{PhilippStout75ASIP}):
\begin{assumplist}[series=assumptions]
    \item $\displaystyle A_n := \sum_{k=1}^{n} a_k^2\to\infty$ as $n\to \infty$; and\label{Assump2:A_n_to_infinite}
    \item $a_n^2 =O (A_n^{1-\delta})$ as $n\to \infty$ for some $\delta \in (0,1]$,\label{Assump2:last_term}
\end{assumplist}
where, for real sequences $\{b_n\}_{n\ge 1}$ and $\{c_n\}_{n\ge 1}$, $b_n=O(c_n)$ as $n\to \infty$ means that $\displaystyle \limsup_{n\to\infty} \frac{|b_n|}{|c_n|} < +\infty$. We also use $b_n \ll c_n$ as $n\to \infty$ to mean that $b_n = O(c_n)$ as $n\to \infty$.
We note that $\{\widetilde{f}_n\}_{n\ge 1}$ is no longer $\{f_n\}_{n\ge 1}$.
Note that we can construct a standard Brownian motion on the Lebesgue probability space due to Chapter IX of Paley and Wiener \cite{PaleyWiener34AMS}. However, $\{\widetilde{f}_n\}_{n\ge 1}$ does not satisfy $\widetilde{f}_{k}(\cdot)-\widetilde{f}_{k-1}(\cdot) = a_k \cos (2\pi n_k\cdot)$ even if we construct a standard Brownian motion on the original Lebesgue probability space $(I, \Borelsigma(I), \mu)$.
In order to preserve the original sequence $\{f_n\}_{n=1}^{\infty}$, we take another Lebesgue probability space $([0,1),\Borelsigma([0,1)),\mu_0)$ on which $\{\widetilde{f}_n\}_{n\ge 1}$ and a standard Brownian motion $\{B(t) \colon t\ge 0\}$ are defined, and construct an enlarged probability space $(I\times [0,1), \Borelsigma(I)\otimes \Borelsigma([0,1)), \nu)$ and a random variable $U$ uniformly distributed over $I$ defined on $(I\times [0,1), \Borelsigma(I)\otimes \Borelsigma([0,1)), \nu)$ such that
\begin{equation}
  f_n(U) = \frac{1}{\sqrt{2}}B(A_n) + o(A_n^{\frac{1}{2}-\lambda}) \quad \text{$\nu$-a.s.,}\label{eq:trigonometric_ASIP_uniform}
\end{equation}
where $\Borelsigma(I)\otimes \Borelsigma([0,1))$ denotes the product $\sigma$-field, and $\nu$ is a coupling measure such that the relation \eqref{eq:trigonometric_ASIP_uniform} holds (see Section \ref{sec:coupling_with_BM} for more details).

In this paper, we establish a precise description of the modulus of continuity of $f_{r,a}(x)$ with the assumptions \eqref{Assump2:A_n_to_infinite} and \eqref{Assump2:last_term}; these assumptions imply \eqref{eq:conti_condition}, see Remark \ref{remark:A2_imply_conti_condition} for the proof. We show that increments of $f_{r,a}(x)$ are approximated by a standard Brownian motion with a sufficiently small error term. We prove this result as an application of the ASIP for the ERWVRP with variable step length defined as follows:
Let $\{X_k\colon k\ge 1\}$ be a $\{+1, -1\}$-valued Markov chain with
\begin{equation}
    P(X_1=+1)=P(X_1=-1)=\frac{1}{2},
\end{equation}
and, for $p\in (0,1)$ and $k\ge 1$,
\begin{equation}
  \begin{dcases}
    P(X_{k+1}=+1\mid X_k = +1) = P(X_{k+1}=-1\mid X_k= -1) = p\\
    P(X_{k+1}=-1\mid X_k = +1) = P(X_{k+1}=+1\mid X_k= -1) = 1-p.
  \end{dcases}
\end{equation}
Let $\{a_k\colon k\ge 1\}$ be a (deterministic) real sequence, and put
\begin{equation}\label{eq:Def_ERWVRP_with_variable_step_length}
  S_0:= 0\quad \text{and} \quad S_n := \sum_{k=1}^n a_k X_k\quad \text{for $n\ge 1$}.
\end{equation}
The stochastic process $\{S_n \colon n\ge 0\}$ is called the ERWVRP with memory parameter $p$ and variable step length.
The process $\{S_n\colon n\ge 0\}$ is a weighted sum of stationary weakly dependent random variables, which has a similar structure to a lacunary trigonometric series with weights defined by \eqref{eq:def_lacunary_series}. Indeed, owing to this structural similarity, we demonstrate that the proof technique by Philipp and Stout \cite{PhilippStout75ASIP} for the ASIP for a lacunary series with weights can be adapted to $\{S_n\colon n\ge 0\}$. Nakano and Takei \cite{NakanoTakeiTakagifunctions} pointed out that letting
\begin{equation}\label{eq:psi_definition}
  \psi_k(x):= \frac{1}{r^{k-1}}d(r^{k-1} x) \quad \text{for $k\ge 1$},
\end{equation}
and $\psi_k^{+}(x)$ be the right-hand derivative of $\psi_k$ at $x$,
\begin{equation}\label{eq:psi_sum}
  w_n (x) := \sum_{k=1}^{n} a_k \psi_k^{+}(x)\quad \text{for $x\in [0,1)$}
\end{equation}
forms the ERWVRP with memory parameter
\begin{equation}
  p_r:=
  \begin{dcases*}
    \frac{1}{2} & for even $r$,\\
    \frac{r+1}{2r} & for odd $r$
  \end{dcases*}
\end{equation}
with variable step length (see also Allaart \cite[Lemma 3.2]{Allaart14}). By the periodicity of $d(x)$, we can redefine $f_{r,a}$ on $\onetorus := \mathbb{R}/\mathbb{Z}$ equipped with the Lebesgue measure $\mu$, and we investigate the modulus of continuity of $f_{r,a}$ at a ``typical'' point with respect to the Lebesgue measure.

 In the same way as trigonometric series, we construct an enlarged probability space on which a random variable $U$ uniformly distributed over $\onetorus$ and a standard Brownian motion $\{B(t) \colon t\ge 0\}$ are defined such that
 \begin{equation}
  f_{r,a}(U+h) - f_{r,a} (U) = h\cdot B(V_{m(h)}) + o(h\cdot V_{m(h)}^{\frac{1}{2}-\lambda}) \quad \text{as $h\to 0$, $\nu$-a.s.}\label{eq:TvdW_ASIP}
\end{equation}
The relation \eqref{eq:TvdW_ASIP} enables us to derive several properties of the increments of $f_{r,a}(x)$ at a ``typical'' point with respect to the Lebesgue measure, from limit theorems for Brownian motions (Corollary \ref{thm:TvdW_class_modulus_of_continuity}).

%
%
The rest of this paper is organized as follows: We present our main results in Section \ref{sec:results}. We investigate growth rates of $\{a_k\}_{k\ge 1}$ under \eqref{Assump2:A_n_to_infinite} and \eqref{Assump2:last_term} in Section \ref{sec:conseq_of_assumptions}. Our results for Takagi--van der Waerden class functions are proved in Sections \ref{sec:Proof_AISPforTvdW} and \ref{sec:proof_modulus_of_continuity}. Sections \ref{sec:ProofofASIP_monotone} and \ref{sec:proof_ASIP} are devoted to the proof of the ASIP for ERWVRP with variable step length.

\section{Results}\label{sec:results}

\subsection{An almost sure invariance principle for the Takagi--van der Waerden class functions}
Let $V_n$ be the variance of $w_n(x)$ defined by \eqref{eq:psi_sum}, i.e.,
\begin{equation}
  V_n := \int_{0}^{1} \bigl(w_n (x)\bigr)^2 \mu(dx).
\end{equation}
Note that $V_n = A_n$ for even $r$. For each $|h| \in (0,1)$, put $m(h):= \floor{\log_r(1/|h|)}$.
The next theorem is the ASIP for the functions $f_{r,a}$ defined by \eqref{eq:TvdWClass} under the assumptions \eqref{Assump2:A_n_to_infinite} and \eqref{Assump2:last_term}. Note that these assumptions ensure that $f_{r,a}$ are continuous.
\begin{theorem}\label{thm:ASIP_TvdW}
  Suppose that \eqref{Assump2:A_n_to_infinite} and \eqref{Assump2:last_term} hold. Then, there exists a probability space $(\onetorus \times [0,1), \Borelsigma(\onetorus)\otimes \Borelsigma([0,1)), \nu)$ on which a random variable $U$ uniformly distributed over $\onetorus$ and a standard Brownian motion $\{B(t)\colon t\ge 0\}$ are defined such that
  \begin{equation}\label{eq:f_U_ASIP}
    f_{r,a}(U+h) - f_{r,a}(U) = h\cdot B\bigl(V_{m(h)}\bigr) + o\bigl(h\cdot V_{m(h)}^{\frac{1}{2}-\lambda}\bigr) \quad \text{as $h\to 0$, $\nu$-a.s.,}
  \end{equation}
  for each $\lambda <\delta/32$.
\end{theorem}
Let $\localmodulusofcont$ be a non-increasing continuous function satisfying
\begin{equation}
  \localmodulusofcont(r^{-n}) = V_n \quad \text{for each $n\in \mathbb{N}$}.
\end{equation}
Note that $\localmodulusofcont(r^{-n}) = A_n$ for even $r$.
\begin{corollary}\label{thm:TvdW_class_modulus_of_continuity}
  Suppose that \eqref{Assump2:A_n_to_infinite} and \eqref{Assump2:last_term} hold. Then, we have:
  \begin{enumerate}[label={(\roman*)}, ref={\roman*}]
    \item \begin{equation}
    \lim_{h\to 0}\mu\Biggl(\Biggl\{x\in \onetorus \colon \frac{ f_{r,a}(x+h)-f_{r,a}(x) }{  h\sqrt{\localmodulusofcont(|h|)} } \le y \Biggr\} \Biggr)= \int_{-\infty}^{y} \frac{1}{\sqrt{2\pi}} e^{-t^2/2}dt.
  \end{equation}
  \label{cor:modulus_of_continuity_CLT}
  \item For $\mu$-a.e.\ $x$, for each $y\in [-1,1]$ there exists a sequence $\{h_n\}_{n=1}^{\infty}$ with $\displaystyle \lim_{n\to \infty} h_n =0$ such that
    \begin{equation}
      \lim_{n\to\infty} \frac{f_{r,a}(x+h_n) -f_{r,a} (x)}{h_n\sqrt{2\localmodulusofcont(|h_n|) \log\log \localmodulusofcont(|h_n|)}} =y.
    \end{equation}
    \label{cor:modulus_of_continuity_LIL}
  \item For $\mu $-a.e.~$x$,
  \begin{equation}
    \liminf_{h \downarrow 0} \sqrt{\frac{\log\log \localmodulusofcont(h)}{\localmodulusofcont(h)}} \sup_{T \in (h, 1)} \frac{|f_{r,a}(x+T)-f_{r,a}(x)|}{T} = \frac{\pi}{\sqrt{8}}.
  \end{equation}
  \label{cor:modulus_of_continuity_otherLIL}
  \item We also assume that the sequence $\{V_n\}_{n\ge 1}$ is regularly varying with index $\beta>0$, i.e., $V_{\floor{nt}} /V_n \to t^{\beta}$ as $n\to \infty$ for all $t\in (0,\infty)$. Then, as $n\to \infty$, for a standard Brownian motion $\{B(t) \colon 0\le t\le 1\}$,
  \begin{equation}
    \biggl\{ \frac{f_{r,a } (x+r^{-\floor{n t^{1/\beta}}}) - f_{r, a}(x)}{r^{-\floor{n t^{1/\beta}}} \cdot \sqrt{V_n} } \colon 0\le t\le 1 \biggr\} \dlimit \{ B(t) \colon 0\le t\le 1\},
  \end{equation}
  where $\dlimit$ denotes convergence in distribution.
  \label{cor:TvdW_weak_convergence}
  \end{enumerate}
\end{corollary}
For $r=2$, Corollary \ref{thm:TvdW_class_modulus_of_continuity} \eqref{cor:modulus_of_continuity_CLT} and part of \eqref{cor:modulus_of_continuity_LIL} were proved in \cite{Allaart09flexibleclass} and \cite{Kono87} by assuming the boundedness of $\{a_k\}_{k\ge 1}$.

\subsection{Almost sure invariance principles for elephant random walks remembering the very recent past with variable step length}
In this section the basic probability space is $(\Omega,\mathcal{F},P)$, and the expectation under $P$ is denoted by $E$. We consider the process $\{S_n\colon n\ge 0\}$ defined by \eqref{eq:Def_ERWVRP_with_variable_step_length}.

Nakano and Takei \cite[Theorem 2.2]{NakanoTakeiTakagifunctions} showed that the stochastic process $\{S_n \colon n\ge 0\}$ diverges with probability one if and only if the condition \eqref{Assump2:A_n_to_infinite} is satisfied. We focus on the detailed asymptotic behavior of the ERWVRP with variable step length $\{S_n\colon n\ge 0\}$.
Let
\begin{equation}\label{eq:S_n_variance}
  s_n = E[(S_n)^2]^{1/2}.
\end{equation}
Note that $s_n^2=A_n $ when $p=1/2$.
The next theorem is the ASIP for $\{S_n\colon n\ge 0\}$.
\begin{theorem}\label{thm:ASIP_non-monotone}
  Suppose that \eqref{Assump2:A_n_to_infinite} and \eqref{Assump2:last_term} hold. Then, for any $p\in (0, 1)$,
  \begin{equation}
    s_n^2 \to \infty \quad \text{as $n\to\infty$},
  \end{equation}
  and without changing the distribution of the process $\{S_n \colon n\ge 0\}$, we can redefine the process $\{S_n\colon n\ge 0\}$ on a richer probability space on which a standard Brownian motion $\{B(t)\colon t\ge 0\}$ is defined such that
\begin{equation}\label{eq:ASIP}
  S_n -  B(s_n^2) = o(s_n^{1-\lambda}) \quad \text{as $n\to \infty$, a.s.}
\end{equation}
 for each $\lambda < \delta/16$.
\end{theorem}
By Theorem \ref{thm:ASIP_non-monotone} and limit theorems for the standard Brownian motion, we can deduce corresponding limit theorems for the process $\{S_n \colon n\ge 0\}$ (for the details see e.g., Section 6.1 and Theorems A--E in Section 1 of Philipp and Stout \cite{PhilippStout75ASIP}). We explicitly state the CLT, LIL and the other LIL for  $\{S_n\colon n\ge 0\}$. 
\begin{corollary}\label{cor:CLT_Strassen_Chung}
  Suppose that \eqref{Assump2:A_n_to_infinite} and \eqref{Assump2:last_term} hold. Then, for any $p\in (0, 1)$, we have the following:
  \begin{enumerate}[label={(\roman*)}, ref={\roman*}, align=left, leftmargin=*]
    \descitem{(CLT)} \begin{equation}
    \lim_{n\to\infty} P\biggl(\frac{S_n}{s_n}\le y\biggr) = \int_{-\infty}^{y} \frac{1}{\sqrt{2\pi}} e^{-t^2/2} dt.
  \end{equation}
    \descitem{(the compact LIL)}\label{cor:Strassen_LIL} With probability one, the set of limit points of
  \begin{equation}
    \frac{S_n}{\sqrt{2 s_n^2 \log\log s_n^2}}
  \end{equation}
  coincides with the closed interval $[-1,1]$.
    \descitem{(the other LIL)}\label{cor:Chung_LIL}
    \begin{equation}
    \liminf_{n\to\infty} \sqrt{\frac{\log\log s_n^2}{s_n^2}}\max_{1\le k\le n} |S_k| = \frac{\pi}{\sqrt{8}} \quad \text{$P$-a.s.}
  \end{equation}
  \end{enumerate}
\end{corollary}
In \cite[Theorem 2.1]{NakanoTakeiTakagifunctions}, the CLT and the classical LIL are obtained for the ERWVRP  (i.e., $a_k\equiv 1$). Note that $A_n = n$ for all $n$ and $\displaystyle s_n^2 \sim \frac{p}{1-p}n$ as $n\to \infty$ in this case. More generally, using the main theorem of \cite{NegishiLILmixingweight}, we have
  \begin{equation}
    \limsup_{n\to \infty} \frac{S_n }{\sqrt{2 \dfrac{p}{1-p} A_n \log\log \biggl(\dfrac{p}{1-p}A_n\biggr)} } = 1 \quad \text{a.s.}\label{eq:exact_LIL_scaled_A_n}
  \end{equation}
  under stronger assumptions on $\{a_k\}_{k\ge 1}$ than \eqref{Assump2:A_n_to_infinite} and \eqref{Assump2:last_term}. We remark that $s_n^2$ in Theorem \ref{thm:ASIP_non-monotone} and Corollary \ref{cor:CLT_Strassen_Chung} cannot be replaced with $\displaystyle \frac{p}{1-p}A_n$ in general. Indeed, noting that $\{(-1)^k X_k\}_{k\ge 1}$ has the same distribution as the increments of the ERWVRP with parameter $1-p$, we have $\displaystyle s_n^2 \sim \frac{1-p}{p}n$ as $n\to\infty$ when $a_k=(-1)^k$.
  It follows from Lemma \ref{lemma:momentL2} below that
  \begin{equation}\label{eq:s_n_square_upperlower}
    \frac{1}{K(p)} A_n \le s_n^2 \le K(p) A_n,
  \end{equation}
  and so we have
  \begin{equation}\label{eq:LIL_scaled_A_n}
    \frac{1}{\sqrt{K(p)}} \le \limsup_{n\to \infty} \frac{S_n }{\sqrt{2A_n \log\log A_n} } \le \sqrt{K(p)} \quad \text{a.s.}
  \end{equation}
  by Corollary \ref{cor:CLT_Strassen_Chung} \eqref{cor:Strassen_LIL}, where
  \begin{equation}
    K(p) := \max \left\{\frac{p}{1-p}, \frac{1-p}{p}\right\}.\label{eq:Kp_definition}
  \end{equation}
  In view of the above discussion, both upper and lower bounds in \eqref{eq:LIL_scaled_A_n} are optimal. Moreover, we can find an example such that limsup in \eqref{eq:LIL_scaled_A_n} is strictly between $\displaystyle \frac{1}{\sqrt{K(p)}}$ and $\sqrt{K(p)}$, see Appendix \ref{appendix:limsup_scaled_A_n}. We give an additional regularity condition to $\{a_k\colon k\ge 1\}$ which enables us to replace $s_n^2$ with $\displaystyle \frac{p}{1-p}A_n$. 
\begin{theorem}\label{thm:monotone}
  Suppose that \eqref{Assump2:A_n_to_infinite} and \eqref{Assump2:last_term} hold. We also assume that
  \begin{assumplist}[resume=assumptions]
    \item $\displaystyle \sum_{k=1}^{\infty} (a_{k+1} - a_k)^2 <+\infty$; or\label{Assump:monotone}
    \item $\{a_k\}_{k\ge 1}$ is a monotone sequence.\label{Assump:monotone_2}
  \end{assumplist}
  Then, for any $p\in (0, 1)$, without changing the distribution of the process $\{S_n \colon n\ge 0\}$, we can redefine the process $\{S_n\colon n\ge 0\}$ on a richer probability space on which a standard Brownian motion $\{B(t)\colon t\ge 0\}$ is defined such that
  \begin{equation}\label{eq:ASIP_monotone}
    S_n - \sqrt{\frac{p}{1-p}} B(A_n) = o(A_n^{ \frac{1}{2} - \frac{\delta}{16} }) \quad \text{a.s.}
  \end{equation}
\end{theorem}
Theorem \ref{thm:monotone} can be applied to $a_k=k^{\beta}$ with $\beta \ge -1/2$, while the main theorem of \cite{NegishiLILmixingweight} is applicable only for $\beta = 0$ or $\beta \ge 1$.

\subsection{Some related results on the classical LIL for weighted sums}
We mention some results related to the classical LIL for weighted sums of random variables. For independent random variables $\{Z_n\}_{n\ge 1}$ with mean $0$ and finite variance, Kolmogorov \cite{Kolmogorov29} showed that, for $S_n=Z_1+\dots+Z_n$, the LIL
\begin{equation}
  \limsup_{n\to \infty} \frac{S_n}{\sqrt{2s_n^2\log\log s_n^2}} = 1 \quad \text{a.s.}
\end{equation}
holds provided that
\begin{Kolmogorovconditions}
  \item $\displaystyle s_n^2 := \sum_{k=1}^{n} E[(Z_k)^2]\to \infty$ as $n\to \infty$; and\label{item:Kolmogorovcondition1}
  \item there exists a sequence $\{K_n\}_{n\ge 1}$ of positive constants with $K_n\to 0$ such that
  \begin{equation}
    |Z_n| \le K_n \frac{s_n}{\sqrt{\log\log s_n^2}} \quad \text{a.s.}
  \end{equation}
  \label{item:Kolmogorovcondition2}
\end{Kolmogorovconditions}
See also Chapter 5 of Stout \cite{Stout74}. Marcinkiewicz and Zygmund \cite{MarcinkiewiczZygmund37LIL} proved that the conditions \eqref{item:Kolmogorovcondition1} and \eqref{item:Kolmogorovcondition2} are necessary for the LIL to hold by constructing a sequence of random variables of the form $Z_n = a_n \varepsilon_n$ satisfying \eqref{item:Kolmogorovcondition1},
\begin{gather}
  |Z_n| \le K \frac{s_n}{\sqrt{\log\log s_n^2}} \quad\text{a.s.}, \quad \text{and}\\
  \limsup_{n\to \infty} \frac{S_n}{\sqrt{2s_n^2\log\log s_n^2}} < 1 \quad \text{a.s.}
\end{gather}
for some constant $K>0$, where $\{a_n\}_{n\ge 1}$ is a deterministic real sequence and $\{\varepsilon_n\}_{n\ge 1}$ are the Rademacher variables (see also Weiss \cite{Weiss59LIL} for counterexamples). We note that, for $\{a_n \varepsilon_n\}_{n\ge 1}$, the condition \eqref{Assump2:A_n_to_infinite} coincides with \eqref{item:Kolmogorovcondition1}, and the condition \eqref{Assump2:last_term} is much stronger than \eqref{item:Kolmogorovcondition2}. For an i.i.d.~sequence $\{X_n\}_{n\ge 1}$ with mean $0$ and variance $1$, Chow and Teicher \cite{ChowTeicher73LILweighted} proved the LIL for weighted sums $S_n=a_1 X_1+\dots +a_nX_n$ under the assumptions \eqref{Assump2:A_n_to_infinite} and
\begin{ChowTeicher}
  \item $\displaystyle \frac{a_n^2}{A_n} \le \frac{C}{n}$ for some $C>0$.\label{item:ChowTeicher}
\end{ChowTeicher}
Note that the condition \eqref{item:ChowTeicher} implies \eqref{Assump2:last_term} by \cite[Lemma 1]{ChowTeicher73LILweighted}. More precise and general results are found in Berkes, H{\"o}rmann and Weber \cite{BerkesHormannWeber10LIL}. For dependent sequences, Negishi \cite{NegishiLILmixingweight} proved the LIL for weighted sums of a stationary $\phi$-mixing sequence of random variables under weaker moment conditions than those of \cite{ChowTeicher73LILweighted}, but stronger regularity conditions on weights. In view of these previous studies, our assumptions \eqref{Assump2:A_n_to_infinite} and \eqref{Assump2:last_term} are not highly restrictive.

\section{Growth rates of weights under \eqref{Assump2:A_n_to_infinite} and \eqref{Assump2:last_term}}\label{sec:conseq_of_assumptions}
In this section, we present some consequences of the assumptions \eqref{Assump2:A_n_to_infinite} and \eqref{Assump2:last_term}. These conditions include the cases $\{a_k\}_{k\ge 1}$ has polynomially growth, but exclude geometric one.

We first note that \eqref{Assump2:last_term} is equivalent to the following statement:
There exists a positive constant $K>0$ such that
\begin{equation}\label{eq:universal_inequality}
  a_n^2 \le  K A_n^{1-\delta}\quad \text{for all $n\ge 1$.}
\end{equation}
In fact, by \eqref{Assump2:last_term}, there exists a positive integer $N$ and a positive constant $C_1$ such that
\begin{center}
  $a_n^2 \le C_1 A_n^{1-\delta}$\quad for all $n\ge N$.
\end{center} 
Let
\begin{equation}\label{eq:K_universal_constant}
  K=\max \biggl\{C_1, \max_{1\le k\le N} \frac{a_k^2}{A_k^{1-\delta}} \biggr\},
\end{equation}
where $0/0:=0$. Then, we obtain
\begin{equation}
  a_k^2 = \frac{a_k^2}{A_k^{1-\delta}} A_k^{1-\delta} \le K A_k^{1-\delta}\quad \text{for all $1\le k\le N$.}
\end{equation}
Thus, we have $a_n^2\le K A_n^{1-\delta}$ for all $n\ge 1$. The universal constant $K>0$ in \eqref{eq:universal_inequality} is fixed for the rest of this paper. It follows from \eqref{eq:universal_inequality} that
\begin{equation}\label{eq:max_a_k_upperbound}
  \max_{1\le k\le n}a_k^2 \le K \max_{1\le k\le n} A_k^{1-\delta} \le K A_n^{1-\delta}.
\end{equation}

By \eqref{eq:universal_inequality}, $A_{n}-A_{n-1} = a_n^2 \le K A_{n}^{1-\delta}$, and hence
\begin{equation}\label{eq:A_n_polynomially_bound}
  A_n \ll n^{1/\delta}.
\end{equation}
Since $A_{n+1} \le A_n + K A_{n+1}^{1-\delta}$ is equivalent to
\begin{equation}
  A_{n+1} \le A_n(1-KA_{n+1}^{-\delta})^{-1},
\end{equation}
by \eqref{Assump2:last_term}, for fixed $q>1$, there exists a positive integer $n_0 = n_0(q)$ such that $A_{n+1} \le A_n q$ for all $n\ge n_0$, and so
\begin{equation}\label{eq:A_j_plus_k_le_A_j_exponential}
  A_{n+k} \le A_n q^k\quad \text{for $n\ge n_0$ and $k\ge 1$.}
\end{equation}

\begin{remark}\label{remark:A2_imply_conti_condition}
  To see that \eqref{eq:conti_condition} follows from \eqref{Assump2:last_term}, note that \eqref{eq:A_n_polynomially_bound} implies
  \begin{equation}
    \frac{|a_n|}{r^{n-1}} \le \frac{\sqrt{K} A_n^{(1-\delta)/2}}{r^{n-1}} \ll \frac{n^{\frac{1-\delta}{2\delta}}}{r^{n-1}},
  \end{equation}
  and $\displaystyle \sum_{n=1}^{\infty} \frac{n^{\frac{1-\delta}{2\delta}}}{r^{n-1}}$ converges for each $\delta\in (0,1]$.
\end{remark}

\section{Proof of Theorem \ref{thm:ASIP_TvdW}}\label{sec:Proof_AISPforTvdW}
Once we obtain Theorem \ref{thm:ASIP_TvdW} for $h\downarrow 0$, we can prove corresponding results for $h\uparrow 0$ by considering $f_{r,a}(1-x)$. Thus, it is sufficient to show Theorem \ref{thm:ASIP_TvdW} for $h\downarrow 0$. In order to prove Theorem \ref{thm:ASIP_TvdW}, we use Theorem \ref{thm:ASIP_non-monotone} which is proved in Section \ref{sec:proof_ASIP}.
 
For each $h\in (0,1/r)$,
there exists a unique integer $m=m(h)$ such that
\begin{equation}
  \frac{1}{r^{m+1}}<h\le \frac{1}{r^{m}}.
\end{equation}
Note that
\begin{equation}
  m(h) \sim \log_r (1/h) \quad \text{as $h\downarrow 0$},
\end{equation}
where $a(h)\sim b(h)$ as $h\to 0$ means that $\displaystyle \lim_{h\to 0} \frac{a(h)}{b(h)}= 1$. Our aim is to obtain the following approximation of the increment $f_{r,a}(x+h)-f_{r,a}(x)$ by the ERWVRP with memory parameter $p_r$ with variable step length, at time $m(h)$:
\begin{equation}\label{eq:increment_approximation}
  f_{r,a}(x+h) - f_{r,a} (x) = h\cdot w_{m(h)}(x) + o(h\cdot (A_{m(h)})^{\frac{1}{2}-\frac{\delta}{8}}) \quad \text{as $h\downarrow 0$ for $\mu$-a.e. $x$,}
\end{equation}
where $\{w_n (x) \colon n\ge 1\}$ is defined by \eqref{eq:psi_sum}. We first prove \eqref{eq:increment_approximation} for even $r$ in Section \ref{sec:even_case}, and then for odd $r$ in Section \ref{sec:odd_case}. In Section \ref{sec:coupling_with_BM}, we construct a probability measure $\nu$ on the measurable space $(\onetorus\times [0,1), \Borelsigma(\onetorus)\otimes \Borelsigma([0,1)))$ such that the relation \eqref{eq:f_U_ASIP} holds.

\subsection{The case $r$ is even}\label{sec:even_case}
For $x\in [0,1)$ and $h\in (0,1/r)$, we consider $r$-ary expansions of $x$ and $x+h$:
\begin{equation}
  x=\sum_{k=0}^{\infty} \frac{\varepsilon_k}{r^{k}}, \quad \text{and} \quad x+h = \sum_{k=0}^{\infty} \frac{\varepsilon_k'}{r^{k}},
\end{equation}
where $\varepsilon_0=0$, $\varepsilon_0'\in \{0,1\}$, $\varepsilon_k\in \{0,1,\dots,r-1\}$ for each positive integer $k$.
Let
\begin{equation}
  k_0=k_0(x,h):=
  \begin{cases*}
    \max \{k\ge 0 \colon \varepsilon_0=\varepsilon_0',\dots,\varepsilon_k=\varepsilon_k'\} & if $\varepsilon_0=\varepsilon_0'$,\\
    -1 & if $\varepsilon_0\neq\varepsilon_0'$.
  \end{cases*}
\end{equation}
Note that $-1\le k_0 \le m$.
If $k\le k_0(x,h)$, then
\begin{equation}
  \psi_k(x+h)-\psi_k(x)= h\cdot \psi_k^{+}(x),
\end{equation}
and therefore we have
\begin{align}
  f_{r,a}(x+h)-f_{r,a}(x) &= h\cdot \sum_{k=1}^{m(h)} a_k \psi_k^{+}(x)\\
  &\quad + \sum_{k=k_0(x,h)+1}^{m(h)} a_k (\psi_k(x+h)-\psi_k(x)-h\cdot \psi_k^{+}(x))\\
  &\quad + \sum_{k=m(h)+1}^{\infty} a_k (\psi_k(x+h)-\psi_k(x)).\label{eq:TvdWclass_decomposition}
\end{align}
Our aim is to show that both the second and third terms are $o(h(A_{m(h)})^{\frac{1}{2}-\frac{\delta}{8}})$ as $h\downarrow 0$.
By the Lipschitz continuity of $\psi_k$,
\begin{equation}
  \left|\sum_{k=k_0(x,h)+1}^{m(h)} a_k (\psi_k(x+h)-\psi_k(x)-h\cdot \psi_k^{+}(x))\right| \le 2h\sum_{k=k_0(x,h)+1}^{m(h)} |a_k|.\label{eq:estimation_second_term}
\end{equation}
The following lemma shows that \eqref{eq:estimation_second_term} is $o(h\cdot A_{m(h)}^{1/2-\delta/8})$.
\begin{lemma}\label{lemma:increments_second_term}
  Suppose that \eqref{Assump2:A_n_to_infinite} and \eqref{Assump2:last_term} hold. Then, for $\mu$-a.e.~$x$,
  \begin{equation}\label{eq:sum_k_0_to_m_absolute_a_k}
    A_{m(h)}^{-1/2+\delta/8}\sum_{k=k_0(x,h)+1}^{m(h)} |a_k| \to 0\quad \text{as $h\downarrow 0$}.
  \end{equation}
\end{lemma}
\begin{proof}
  The proof is along the same procedure as the argument in \cite[Lemma 5.3]{Allaart09flexibleclass}, however we do not assume the boundedness of $\{a_k\}_{k\ge 1}$.
  By the Cauchy--Schwarz inequality, in order to prove \eqref{eq:sum_k_0_to_m_absolute_a_k}, it is enough to show that 
  \begin{equation}
    \frac{(A_{m(h)}-A_{k_0(x,h)}) (m(h)-k_0(x,h))}{A_{m(h)}^{1-\delta/4}}\to 0 \quad \text{as $h\downarrow 0$}
  \end{equation}
  for $\mu$-a.e.~$x$. Noting that $k_0(x,h)$ is non-increasing in $h$ for fixed $x$, it is sufficient to consider the case $h=r^{-\ell}$ for $\ell \in\mathbb{N}$. By Equation (4.7) of \cite{NakanoTakeiTakagifunctions}, we obtain
  \begin{align}
    E[(A_{\ell} -A_{k_0(x, r^{-\ell})})^2(\ell-k_0(x,r^{-\ell}))^2] &\le \sum_{j=1}^{\ell+1} (A_{\ell}-A_{\ell-j})^2 j^2 r^{-(j-1)}\quad \text{for $\ell \in \mathbb{N}$},
  \end{align}
  and hence it follows from Chebyshev's inequality that, for any fixed $\varepsilon>0$,
  \begin{align}
    &\sum_{\ell=1}^{\infty} \mu\biggl(\biggl\{ x\in \onetorus \colon \frac{ \bigl( A_{\ell} - A_{k_0(x,r^{-\ell})} \bigr) (\ell - k_0(x,r^{-\ell})) }{A_{\ell}^{1-\delta/4} } \ge \varepsilon \biggr\}\biggr)\\
    &\le \sum_{\ell=1}^{\infty} \frac{E[(A_{\ell} -A_{k_0(x, r^{-\ell}) })^2 (\ell  - k_0(x,r^{-\ell}))^2  ]}{\varepsilon^2 A_{\ell}^{2-\delta/2}}\\
    &\le \frac{r}{\varepsilon^2} \sum_{\ell=1}^{\infty} \sum_{j=1}^{\ell +1} \frac{j^2}{r^{j}}\cdot \frac{(A_{\ell} - A_{\ell -j })^2 }{A_{\ell}^{2-\delta/2}}\\
    &=\frac{r}{\varepsilon^2} \sum_{j=1}^{\infty} \frac{j^2}{r^{j}} \sum_{\ell=j}^{\infty} \frac{(A_{\ell} - A_{\ell -j })^2 }{A_{\ell}^{2-\delta/2}} + \frac{r}{\varepsilon^2} \sum_{\ell=1}^{\infty} \frac{(\ell+1)^2}{r^{\ell+1}} A_{\ell}^{-\delta/2}.\label{eq:A_m_minus_A_k_0_complete_convergence} 
  \end{align}
  By the Cauchy--Schwarz inequality, we have
  \begin{equation}
    (A_{\ell} - A_{\ell -j})^2 \le j\sum_{k=\ell -j +1}^{\ell} a_{k}^4 =j \sum_{k=0}^{j-1} a_{\ell -k}^4.
  \end{equation}
  Thus, for $j\ge 1$, by \eqref{eq:universal_inequality},
  \begin{align}
    \sum_{\ell=j}^{\infty} \frac{(A_{\ell} - A_{\ell -j })^2 }{A_{\ell}^{2-\delta/2}} &\le j \sum_{\ell=j}^{\infty} \sum_{k=0}^{j-1} \frac{a_{\ell - k }^4}{A_{\ell}^{2-\delta/2}} = j\sum_{k=0}^{j-1} \sum_{\ell=j}^{\infty} \frac{a_{\ell - k}^4}{A_{\ell}^{2-\delta/2}}\\
    &= j\sum_{k=0}^{j-1} \sum_{\ell = j-k}^{\infty} \frac{a_{\ell}^4}{A_{\ell+k}^{2-\delta/2}}\le j\sum_{k=0}^{j-1} \sum_{\ell=1}^{\infty} \frac{a_{\ell}^4}{A_{\ell}^{2-\delta/2}}\\
    &\le Kj^2 \sum_{\ell =1}^{\infty} \frac{a_{\ell}^{2}}{A_{\ell}^{1+\delta/2}}.
  \end{align}
  By the Abel--Dini--Pringsheim theorem (Theorem \ref{thm:AbelDiniPringsheim} in Appendix \ref{appendix:some_technical_results}), the series $\displaystyle \sum_{\ell =1}^{\infty} \frac{a_{\ell}^{2}}{A_{\ell}^{1+\delta/2}}$ converges, and we see that \eqref{eq:A_m_minus_A_k_0_complete_convergence} converges.
  By the Borel--Cantelli Lemma,
  \begin{equation}
    \lim_{\ell \to \infty} \frac{ \bigl( A_{\ell} - A_{k_0(x,r^{-\ell})} \bigr) (\ell - k_0(x,r^{-\ell})) }{A_{ \ell }^{1-\delta/4} } = 0 \quad \text{for $\mu$-a.e.~$x$}.
  \end{equation}
\end{proof}
For the third term of \eqref{eq:TvdWclass_decomposition}, we have
\begin{equation}
  \left|\sum_{k=m(h)+1}^{\infty} a_k (\psi_k(x+h) -\psi_k(x))\right| \le 2\sum_{k=m(h)+1}^{\infty} \frac{|a_k|}{r^{k-1}}.\label{eq:estimation_third_term}
\end{equation}
We show \eqref{eq:estimation_third_term} is $O(h\cdot A_{m(h)}^{1/2-\delta/2})$.
\begin{lemma}\label{lemma:increments_third_term}
  Suppose that \eqref{Assump2:A_n_to_infinite} and \eqref{Assump2:last_term} hold. Then,
  \begin{equation}\label{eq:increments_third_term}
    \sum_{k=m+1}^{\infty} \frac{|a_k|}{r^{k-1}} = O(h\cdot A_{m(h)}^{1/2-\delta/2})\quad \text{as $h\downarrow 0$}.
  \end{equation}
\end{lemma}
\begin{proof}
  We write $m=m(h)$.
  Since
  \begin{equation}\label{eq:tail_sum}
    \sum_{k=m+1}^{\infty} \frac{|a_k|}{r^{k-1}} =\frac{r^2}{r^{m+1}} \sum_{k=1}^{\infty} \frac{|a_{m+k}|}{r^{k}}\le r^2h\sum_{k=1}^{\infty} \frac{|a_{m+k}|}{r^{k}},
  \end{equation}
  it is sufficient to show $\displaystyle \sum_{k=1}^{\infty} \frac{|a_{m+k}|}{r^{k}}=O(A_{m}^{1/2-\delta/2})$.

  Applying \eqref{eq:A_j_plus_k_le_A_j_exponential} to $q= r^{\frac{1}{1-\delta}} >1$, there exists a positive integer $n_0$ such that
  \begin{equation}
    A_{m+k}^{(1-\delta)/2} \le A_m^{(1-\delta)/2} r^{k/2}
  \end{equation}
  for all $m\ge n_0$ and $k\ge 1$, and hence we have
  \begin{equation}
    \frac{|a_{m+k}|}{r^{k}} \le \frac{K A_{m+k}^{(1-\delta)/2}}{r^{k}} \le KA_m^{(1-\delta)/2} r^{k/2}
  \end{equation}
  by \eqref{eq:universal_inequality}.
  Thus, we obtain
  \begin{equation}
    \sum_{k=1}^{\infty} \frac{|a_{m+k}|}{r^{k}} \le K A_{m}^{(1-\delta)/2}\sum_{k=1}^{\infty} \frac{1}{r^{k/2}}  
  \end{equation}
  for all $m\ge n_0$, which implies \eqref{eq:increments_third_term}.
\end{proof}
Combining \eqref{eq:estimation_second_term}, \eqref{eq:estimation_third_term} and \eqref{eq:TvdWclass_decomposition} yields \eqref{eq:increment_approximation}.

\subsection{The case $r$ is odd}\label{sec:odd_case}
Let
\begin{equation}
  \widehat{k}_0 = \widehat{k}_0(x,h) := k_0\biggl(x+\frac{1}{2},h\biggr),
\end{equation}
and ${k_0}\wedge {\widehat{k}_0} = ({k_0}\wedge {\widehat{k}_0})(x,h) := \min\{{k_0}(x,h),  {\widehat{k}_0}(x,h)\}$. Then we have
\begin{equation}
\label{eq:LinearOddk0Hatk0}
\psi_k(x+h)-\psi_k(x)=\psi_k^+ (x) \cdot h \quad \mbox{if $k \leq (k_0 \wedge \widehat{k}_0)(x,h) $.}
\end{equation}
By the proof of Lemma 4.3 of \cite{NakanoTakeiTakagifunctions},
\begin{equation}
  \mu\bigl(\ell - (k_0\wedge \widehat{k}_0) \ge j \bigr) \le 2r^{-(j-1)}\quad \text{for $j\in\mathbb{N}$}.
\end{equation}
Replacing $k_0(x,h)$ with $\widehat{k}_0 (x,h)$, Lemma \ref{lemma:increments_second_term} and \ref{lemma:increments_third_term} also hold for odd $r$, and so we obtain \eqref{eq:increment_approximation}.

\subsection{Coupling with a standard Brownian motion}\label{sec:coupling_with_BM}
We next construct a probability space $(\onetorus\times [0,1), \Borelsigma(\onetorus)\otimes \Borelsigma([0,1)), \nu)$ such that the relation \eqref{eq:f_U_ASIP} holds. Our construction is inspired by Chapter 6 of Kallenberg \cite{Kallenberg2nd}, an extension of a probability space. Note that we can construct a standard Brownian motion on the Lebesgue probability space $([0,1), \mathcal{B}([0,1)), \mu_0)$ due to Chapter IX of Paley and Wiener \cite{PaleyWiener34AMS}. In order to distinguish the Lebesgue measure on $[0,1)$ from $\mu$ on $\onetorus$, we write the Lebesgue measure on $([0,1), \mathcal{B}([0,1)))$ for $\mu_0$. By Theorem \ref{thm:ASIP_non-monotone}, there exists a sequence $\widetilde{Y} = \{\widetilde{Y}_k\}_{k=1}^{\infty}$ of random variables and a standard Brownian motion $\{W(t)\colon t\ge 0\}$ on $([0,1), \mathcal{B}([0,1)), \mu_0)$ such that $\widetilde{Y}$ has the same distribution as $\Psi = \{a_k \psi_k^{+}\}_{k=1}^{\infty}$, and, for each $\lambda<\delta/32$,
\begin{equation}
  \sum_{k=1}^{n} \widetilde{Y}_k - W(V_n) = o(V_n^{1/2-\lambda}) \quad \text{as $n\to \infty$, $\mu_0$-a.s.,}\label{eq:tilde_Y_ASIP}
\end{equation}
where $\mu_0$ is the Lebesgue measure.

For $t\in [0,\infty)$ and $(x, \omega) \in \onetorus \times [0,1)$, we set
\begin{equation}
  U(x,\omega) := x,\quad Y_k(x, \omega) := \widetilde{Y}_k (\omega), \quad \text{and}\quad B(t)(x, \omega) := W(t)(\omega).
\end{equation}
Let $Y=\{Y_k\}_{k=1}^{\infty}$, and $Q$ be the distribution of $\widetilde{Y}$. Note that $Q$ is also equal to the distribution of $\Psi$. We need to construct a probability measure $\nu$ on $(\onetorus\times [0,1), \Borelsigma(\onetorus)\otimes \Borelsigma([0,1)))$ such that
\begin{enumerate}[label={(\roman*)}, ref={\roman*}]
  \item $\Psi(U) = Y$ $\nu$-a.s.,\label{item:Psi_equal_Y}
  \item its marginal on $\onetorus$ is $\mu$; and\label{item:Psi_first_marginal}
  \item its marginal on $[0,1)$ is $\mu_0$.\label{item:Psi_second_marginal}
\end{enumerate}
Since the countable product $\seqspace$ of the reals is a complete separable metric space, there exists the regular conditional probability measure $\mu_0(\cdot \mid \widetilde{Y}=z)$ on $\Borelsigma([0,1))$ for each $z\in \seqspace$ (see e.g., Kallenberg \cite[Theorem 6.3]{Kallenberg2nd}). The measure $\mu_0(\cdot \mid \widetilde{Y}=z)$ is concentrated on the set $(\widetilde{Y})^{-1} (\{z\})$, i.e., $\mu_0((\widetilde{Y})^{-1} (\{z\}) \mid \widetilde{Y}=z) = 1$ for $Q$-a.e. $z$. For $E\in \Borelsigma(\onetorus)\otimes \Borelsigma([0,1))$, let
\begin{equation}
  \nu (E) := \int_{\onetorus}^{} \mu(dx) \mu_0( E_x \mid \widetilde{Y}=\Psi(x)),
\end{equation}
where $E_x := \{\omega \in [0,1) \colon (x,\omega) \in E\}$. We show that the measure $\nu$ satisfies \eqref{item:Psi_equal_Y}, \eqref{item:Psi_first_marginal} and \eqref{item:Psi_second_marginal}. Then, for any $A\in \Borelsigma(\onetorus)$,
\begin{align}
  \nu (A\times [0,1)) &= \int_{A}^{} \mu(dx) \mu_0([0,1) \mid \widetilde{Y}=\Psi(x) )\\
  &=\int_{A}^{} \mu(dx) =\mu(A).
\end{align}
The regular conditional probability measure $\mu_0(\cdot \mid \widetilde{Y}=z )$ satisfies
\begin{equation}
  \mu_0((\widetilde{Y})^{-1}(A) \cap B ) = \int_{A}^{} Q(dz) \mu_0(B \mid \widetilde{Y}=z)\label{eq:regular_probability_integration}
\end{equation}
for any $A\in \Borelsigma(\seqspace)$ and $B\in \Borelsigma([0,1))$. Thus, since $Q(dz)=(\mu\circ \Psi^{-1})(dz)$, we have, for any $B\in \Borelsigma([0,1))$,
\begin{align}
  \nu(\onetorus \times B) &= \int_{\onetorus}^{} \mu(dx) \mu_0(B\mid \widetilde{Y}=\Psi(x))\\
  &=\int_{\seqspace}^{} Q(dz) \mu_0( B\mid \widetilde{Y}=z)=\mu_0(B).
\end{align}
We show that \eqref{item:Psi_equal_Y} holds. Noting that $\mu_0(\cdot \mid \widetilde{Y}=z)$ is unique up to $Q$-a.e.~$z$ and $Q=\mu\circ \Psi^{-1}$, we have
\begin{equation}
  \mu_0( (\widetilde{Y})^{-1}(\Psi(x))  \mid \widetilde{Y} = \Psi(x)) =1 \quad \text{for $\mu$-a.e.~$x$.}
\end{equation}
Therefore, letting
\begin{equation}
  E= \{(x,\omega) \in \Omega_0 \colon \Psi(U(x,\omega)) = Y(x,\omega)\},
\end{equation}
we have
\begin{align}
  \nu (E) &= \int_{\onetorus}^{} \mu(dx) \mu_0(\{\omega \in [0,1) \colon \widetilde{Y}(\omega)=\Psi(x)\} \mid \widetilde{Y} = \Psi(x))\\
  &=\int_{\onetorus}^{} \mu(dx)\mu_0( (\widetilde{Y})^{-1}(\Psi(x))  \mid \widetilde{Y} = \Psi(x)) = 1,
\end{align}
which means
\begin{equation}
  \Psi(U)=Y \qquad \text{$\nu$-a.s.}
\end{equation}
Hence, by \eqref{eq:increment_approximation} and \eqref{eq:tilde_Y_ASIP}, we have, $\nu$-a.s.,
\begin{align}
  f_{r,a}(U+h) -f_{r,a}(U) &= h\cdot \sum_{k=1}^{m(h)} a_k\psi_k^{+} (U) + o(h \cdot A_{m(h)}^{1/2-\delta/8})\quad \text{as $h\downarrow 0$, and}\\
  \sum_{k=1}^{m(h)} a_k\psi_k^{+} (U) &= B(V_{m(h)}) + o(V_{m(h)}^{1/2-\lambda}) \quad \text{as $h\downarrow 0$},
\end{align}
which implies Theorem \ref{thm:ASIP_TvdW} by using \eqref{eq:s_n_square_upperlower}.

\section{Proof of Corollary \ref{thm:TvdW_class_modulus_of_continuity}}\label{sec:proof_modulus_of_continuity}
Note that $\sigma_l(|h|) \sim V_{m(h)}$ as $h\to 0$.

\eqref{cor:modulus_of_continuity_CLT} By Theorem \ref{thm:ASIP_TvdW}, we have
\begin{equation}
  \frac{f_{r,a}(U+h) -f_{r,a}(U)}{h\cdot \sigma_l(|h|)} \dlimit N \quad \text{as $h\to 0$},
\end{equation}
where $N$ is a standard normal random variable. Since $\nu\circ U^{-1}$ is the Lebesgue measure, we have \eqref{cor:modulus_of_continuity_CLT}.

\eqref{cor:modulus_of_continuity_LIL} By the LIL for Brownian motions, for $\nu$-a.s., for each $y\in [-1,1]$ there exists a sequence $\{h_n\}_{n=1}^{\infty}$ with $\lim_{n\to \infty} h_n =0$ such that
\begin{equation}
  \lim_{n\to \infty } \frac{B(V_{m(h_n)})}{\sqrt{2V_{m(h_n)} \log\log V_{m(h_n)}}} = y,
\end{equation}
which implies \eqref{cor:modulus_of_continuity_LIL}.

\eqref{cor:modulus_of_continuity_otherLIL} By Chung's LIL for Brownian motions, we have
\begin{equation}
  \liminf_{h\to 0} \sqrt{\frac{\log\log V_{m(h)}}{V_{m(h)}}}\sup_{0\le s\le V_{m(h)}} |B(s)| = \frac{\pi}{\sqrt{8}} \quad \text{$\nu$-a.s.}
\end{equation}
Note that, for any $h\in (0,1)$, $\nu$-a.s.,
\begin{align}
  \sup_{T\in (h, 1)}\left|\frac{f_{r,a}(U+T) - f_{r,a}(U)}{T} - B(V_{m(T)})\right| &\le \sup_{T\in (h,1)} V_{m(T)}^{1/2-\lambda}\\
  &\le K(p) A_{m(h)}^{1/2-\lambda}
\end{align}
by \eqref{eq:s_n_square_upperlower}, which implies \eqref{cor:modulus_of_continuity_otherLIL}.

\eqref{cor:TvdW_weak_convergence} Taking $h=r^{-\floor{n t^{1/\beta}}}$ in \eqref{eq:f_U_ASIP}, we have
\begin{equation}
  \frac{f_{r}(U + r^{-\floor{n t^{1/\beta}}}) -f_{r}(U)}{r^{-\floor{n t^{1/\beta}}}}  = B(V_{\floor{nt^{1/\beta}}}) + o(V_{\floor{n t^{1/\beta}}}^{\frac{1}{2}-\lambda}) \quad \text{as $n\to\infty$, $\nu$-a.s.}
\end{equation}
Let $V_0 := 0$. By the uniform convergence theorem for regularly varying functions (see e.g., \cite[Theorem 1.5.2]{Bingham87RegularVariation}),
\begin{equation}
  \delta_n := \sup_{0\le t\le 1} \left|\frac{V_{\floor{nt^{1/\beta}}}}{V_n} -t \right| \to 0 \quad \text{as $n\to \infty$},
\end{equation}
and so
\begin{align}
\frac{1}{\sqrt{V_n}}\sup_{0\le t\le 1} |B(V_n t) - B(V_{\floor{nt^{1/\beta}}}) | &\le \frac{1}{\sqrt{V_n}}\sup_{0\le u, v\le 1; |u-v|\le \delta_n} |B(u) -B(v) |\\
&\to 0 \quad \text{in $\nu$-probability.}
\end{align}
Thus, we obtain
\begin{equation}
  \sup_{0\le t\le 1} \left|\frac{f_{r}(U + r^{-\floor{n t^{1/\beta}}}) -f_{r}(U)}{r^{-\floor{n t^{1/\beta}}}\cdot \sqrt{V_n} } - \frac{B( V_n t )}{\sqrt{V_n}}\right| \to 0 \quad \text{in $\nu$-probability,}
\end{equation}
which proves \eqref{cor:TvdW_weak_convergence}.

\section{Proof of Theorem \ref{thm:monotone}}\label{sec:ProofofASIP_monotone}
Our proof of Theorem \ref{thm:monotone} is much easier than that of Theorem \ref{thm:ASIP_non-monotone}, and so we give it first. The method is inspired by \cite[Theorem 2.1]{Shao93ASIP}.

We start by presenting some basic properties of the stochastic process $\{S_n \colon n\ge 0\}$ defined by \eqref{eq:Def_ERWVRP_with_variable_step_length}.
Let $\mathcal{F}_0$ be the trivial $\sigma$-field, and for $n \geq 1$, let $\mathcal{F}_n = \sigma\{X_1, X_2, \ldots, X_n\}$ be the $\sigma$-field generated by $X_1, X_2, \ldots, X_n$. Then, we have
\begin{equation}\label{eq:step_conditional_exp}
  E[X_{n+1} \mid \mathcal{F}_n] = \alpha X_n,
\end{equation}
and
\begin{equation}
  E[X_k X_l] = \alpha^{|l-k|}
\end{equation}
for any $k,l\ge 1$, where $\alpha=2p-1$.
\begin{lemma}[{Nakano and Takei \cite[Section 3]{NakanoTakeiTakagifunctions}}]
  The sequence $\{X_k\colon k\ge 1\}$ is a $\phi$-mixing sequence with
  \begin{equation}\label{eq:phi_definition}
    \phi(m) : = \frac{|\alpha|^m}{2}.
\end{equation}
\end{lemma}
We frequently use the following moment inequality for $S_n$.
\begin{lemma}[{Nakano and Takei \cite[Lemma 3.1]{NakanoTakeiTakagifunctions}}]\label{lemma:momentL2}
  For any $\{a_k\}_{k\ge 1}$ and any integers $n>m\ge 1$, we have
  \begin{equation}
    \frac{1}{K(p)}\cdot (A_n-A_m) \le E\bigl[(S_m-S_n)^2\bigr] \le K(p)\cdot (A_n-A_m).
  \end{equation}
  Here $K(p)$ is the constant defined by \eqref{eq:Kp_definition}.
\end{lemma}
We use the Doob decomposition to prove Theorem \ref{thm:monotone}. By \eqref{eq:step_conditional_exp}, $d_k:= X_k - \alpha X_{k-1}$ forms a martingale difference sequence, where $\alpha=2p-1$. Letting $M_n := \sum_{k=1}^{n} a_k d_k$, we find that $\{M_n\}_{n=1}^{\infty}$ is a zero-mean and square-integrable martingale, and we have
\begin{equation}\label{eq:mart_repr}
  S_n - \frac{1}{1-\alpha} M_n = \frac{\alpha}{1-\alpha} \sum_{k=1}^{n-1} (a_{k+1} - a_k ) X_k -\frac{\alpha a_n X_n}{1-\alpha}.
\end{equation}

  Applying the Skorokhod embedding theorem (Lemma \ref{lemma:Skorokhodembeddingtheorem} in Appendix \ref{appendix:some_technical_results}) to $\{M_n\}_{n=1}^{\infty}$, there exists a standard Brownian motion $\{B'(t) \colon t\ge 0\}$ and a non-decreasing sequence $\{T_n\}_{n=1}^{\infty}$ of random variables such that \eqref{eq:embedding}, \eqref{eq:Skorokhod_tau_n_expectation} and \eqref{eq:Skorokhod_tau_n_square_expectation} hold.
  Since $|d_n|\le 2$, we have
  \begin{align}
    \sum_{k=1}^{\infty} \frac{E[ (\tau_k - E[\tau_k \mid \mathcal{F}_{k-1}])^2 \mid \mathcal{F}_{k-1} ]}{\bigl(A_k^{1-\frac{\delta}{4}}\bigr)^2} &\le 16 C_1 \sum_{k=1}^{\infty} \frac{a_k^4}{A_k^{2-\frac{\delta}{2}}}\\
     &\le 16 C_1 K \sum_{k=1}^{\infty} \frac{a_k^2}{A_k^{1+\frac{\delta}{2}}} <+\infty \quad \text{a.s.}
  \end{align}
  by \eqref{eq:Skorokhod_tau_n_square_expectation}, \eqref{eq:universal_inequality} and Theorem \ref{thm:AbelDiniPringsheim}, and so it follows from \eqref{eq:Skorokhod_tau_n_expectation}, Doob's convergence theorem (Corollary 2.2 in \cite{HallHeydemart}), and Kronecker's lemma that
  \begin{equation}\label{eq:tau_LLN}
    T_n - (1-\alpha^2) A_n = \sum_{k=1}^{n} (\tau_k - E[\tau_k \mid \mathcal{F}_{k-1}]) = o(A_n^{1-\frac{\delta}{4}})\quad \text{a.s.}
  \end{equation}
  Letting $\displaystyle B(t):= \frac{1}{\sqrt{1-\alpha^2}} B'((1-\alpha^2) t)$, $\{B(t)\colon t\ge 0\}$ is also a standard Brownian motion, and we have, a.s.,
  \begin{equation}
    M_n = B' (T_n) = B'((1-\alpha^2) A_n) + o(A_n^{\frac{1}{2}-\frac{\delta}{16}}) = \sqrt{1-\alpha^2} B(A_n) + o(A_n^{\frac{1}{2}-\frac{\delta}{16}})
  \end{equation}
  by \eqref{eq:tau_LLN} and Theorem 3.2.A in \cite{HansonRusso83}. We show the right-hand side of \eqref{eq:mart_repr} is $o(A_n^{\frac{1}{2}-\frac{\delta}{16}})$. By the assumption \eqref{Assump:monotone} and \cite[Theorem 2.2]{NakanoTakeiTakagifunctions}, the first term on the right-hand side of \eqref{eq:mart_repr} converges a.s., and hence it is $o(A_n^{\frac{1}{2}-\frac{\delta}{16}})$. Since
  \begin{equation}\label{eq:a_n_negligible_compared_to_A_n_one-half}
    \frac{|a_n|}{A_n^{\frac{1}{2}-\frac{\delta}{16}}} \le \frac{\sqrt{K}}{A_n^{\frac{7\delta}{16}}}  \to 0
  \end{equation}
  by \eqref{eq:universal_inequality} and \eqref{Assump2:A_n_to_infinite}, the second term is $o(A_n^{\frac{1}{2}-\frac{\delta}{16}})$. If \eqref{Assump:monotone_2} holds, the first and second terms are also $o(A_n^{\frac{1}{2}-\frac{\delta}{16}})$ by \eqref{eq:a_n_negligible_compared_to_A_n_one-half}.
  By \eqref{eq:mart_repr}, we see that the relation \eqref{eq:ASIP_monotone} holds. This completes the proof of Theorem \ref{thm:monotone}.

\section{Proof of Theorem \ref{thm:ASIP_non-monotone}}\label{sec:proof_ASIP}
This section is devoted to the proof of Theorem \ref{thm:ASIP_non-monotone}. If $p=1/2$, then the process $\{S_n\colon n\ge 0\}$ is a zero-mean and square-integrable martingale, and Theorem \ref{thm:ASIP_non-monotone} is a direct consequence of Theorem 1.3 of Strassen \cite{Strassen67}. In the rest of this section, we always assume that $p\neq 1/2$.

\subsection{Proof outline}\label{sec:proof_outline}
 
 We begin with a brief outline of the proof of Theorem \ref{thm:ASIP_non-monotone}. Our proof is inspired by Chapter 6 of Philipp and Stout \cite{PhilippStout75ASIP} and that of Theorem 4 in Berkes and Philipp \cite{BerkesPhilipp79}. It is a combination of a blocking procedure, a Gordin type martingale approximation, and the Skorokhod embedding theorem.

 First, we partition the index set into disjoint blocks $I_j$ and define the block sums $Y_j:=\sum_{k\in I_j} a_k X_k$. Although the variance of $S_n$ is not easy to analyze directly, a suitable choice of blocks gives good control of the variance $\sigma_j^2$ of $Y_j$, and we can obtain a law of large numbers for $\{Y_j^2\}$.

By the Doob decomposition, we can decompose $\sum_{j}^{} Y_j$ as the sum of the martingale difference $ \sum_{j}^{} (Y_j-E[Y_j\mid Y_1,\dots,Y_{j-1}])$ and the predictable term $\sum_{j}^{} E[Y_j\mid Y_1,\dots,Y_{j-1}]$, but in the proof of Theorem \ref{thm:monotone} above, it is not obvious whether the predictable term is negligible under \eqref{Assump2:A_n_to_infinite} and \eqref{Assump2:last_term} only.
Therefore, we use a Gordin type decomposition, namely
\begin{equation}\label{eq:Gordin_decomposition}
  Y_j = \xi_j +u_j -u_{j+1},
\end{equation}
where $\{\xi_j\}$ is a martingale difference sequence, and
\begin{equation}
  u_j := \sum_{\ell=0}^{\infty} E[Y_{j+\ell} \mid Y_1,\dots,Y_{j-1}].\label{eq:u_j_definition}
\end{equation}
Since
\begin{equation}
  \sum_{j=1}^{M} Y_j = \sum_{j=1}^{M} \xi_j + u_1 - u_{M+1},
\end{equation}
it is enough to show $u_j$ is negligible compared to the suitable scaling for the martingale.
For general results on Gordin type martingale approximations, see Theorem 4.1 of Merlev{\`e}de, Peligrad and Utev \cite{MerlevedePeligradUtev19DependentStructures}.

In the final step, applying the Skorokhod embedding theorem (Lemma \ref{lemma:Skorokhodembeddingtheorem}) to $\Delta M_n = \xi_n$, there exists a standard Brownian motion $\{B(t) \colon t\ge 0\}$ and a non-decreasing sequence $\{T_n\}_{n=1}^{\infty}$ of random variables satisfying \eqref{eq:embedding}, \eqref{eq:Skorokhod_tau_n_expectation}, and \eqref{eq:Skorokhod_tau_n_square_expectation}.
The error term in an approximation of $S_n$ by $B(s_n^2)$ is evaluated via the law of large numbers for $\{\tau_n\}_{n=1}^{\infty}$, where $\tau_n = T_n-T_{n-1}$.

\subsection{Introduction of the blocks}
Let $h_1=0$ and $A_0=0$. We define $h_{n+1}$ inductively as the smallest integer $h$ with
\begin{equation}\label{eq:definition_h_n}
  A_{h} - A_{h_n} \ge A_{h_n}^{1-\delta/2}.
\end{equation}
To avoid sub-subscript, we write $B_n=A_{h_n}$.

By the minimality of $h_{n+1}$ and \eqref{eq:universal_inequality},
\begin{equation}\label{eq:B_n_diff_upper_BD_1}
  B_n^{1-\delta/2} \le B_{n+1} -B_{n} = A_{h_{n+1}-1} - B_n + a_{h_{n+1}}^2\le B_n^{1-\delta/2} + KB_{n+1}^{1-\delta}.
\end{equation}
Using \eqref{Assump2:A_n_to_infinite},
\begin{equation}
  0 \le \frac{B_n^{1-\delta/2}}{B_{n+1}} = \frac{B_n}{B_{n+1}} \cdot B_{n}^{-\delta/2} \to 0 \quad \text{as $n\to\infty$},
\end{equation}
which together with \eqref{eq:B_n_diff_upper_BD_1} implies that
\begin{equation}\label{eq:B_nPlus1_sim_B_n}
  B_{n+1} -B_n \sim B_{n}^{1-\delta/2} \quad \text{and} \quad B_{n+1}\sim B_{n} \quad \text{as $n\to\infty$}.
\end{equation}
Here $b_n\sim c_n$ as $n\to \infty$ means that $\lim_{n\to \infty} b_n/c_n =1$.

For $j\in \mathbb{N}$, let
\begin{equation}
  I_j := (h_j, h_{j+1}]\cap \mathbb{Z},\quad Y_j := \sum_{k\in I_j} a_k X_k,\quad \text{and} \quad \sigma_j:=E[(Y_j)^2]^{1/2}.
\end{equation}
Recall that $s_n^2$ is defined by \eqref{eq:S_n_variance}. We will show that $s_{h_{j+1}}^2 - s_{h_{j}}^2$ is approximated by $\sigma_j^2$.
\begin{lemma}\label{lemma:s_h_j_asymp_sigma_j_B_j}
  As $j\to \infty$, $s_{h_j}^2 \asymp \sigma_j^2 B_j^{\delta/2}$, where $b_n\asymp c_n$ as $n\to \infty$ means that $b_n=O(c_n)$ and $c_n=O(b_n)$ as $n\to \infty$.
\end{lemma}
\begin{proof}
  By \eqref{Assump2:A_n_to_infinite} and Lemma \ref{lemma:momentL2},
  \begin{equation}\label{eq:asymptotics_s_h_j_sigma_j}
    s_{h_j}^2 \asymp B_j,\quad \text{and}\quad \sigma_j^2 \asymp (B_{j+1}-B_{j}).
  \end{equation}
  Then, by \eqref{eq:definition_h_n} and \eqref{eq:B_nPlus1_sim_B_n},
  \begin{equation}
    \frac{s_{h_j}^2}{\sigma_j^2} \asymp \frac{B_j}{B_{j+1}-B_j} \asymp B_j^{\delta/2}.
  \end{equation}
\end{proof}
\begin{lemma}\label{lemma:s_h_j_diff_LLN}
  We have
  \begin{align}
    \frac{s_{h_{j+1}}^2 -s_{h_j}^2}{\sigma_j^2} &= 1+ O(B_j^{-\delta/4}\cdot (\log B_j)^{1/2}) \quad \text{as $j\to \infty$; and}\label{eq:s_h_j_LLN}\\
    s_{h_{M+1}}^2 -\sum_{j=1}^{M}\sigma_j^2 &= O(B_M^{1-\delta/8}\cdot (\log B_M)^{1/2} ) \quad \text{as $M\to \infty$.}\label{eq:s_h_Mplus1_sum_sigma_j}
  \end{align}
\end{lemma}
\begin{proof}
  Using Minkowski's inequality,
  \begin{align}
    |s_{h_{j+1}} - s_{h_{j}}| = | E[ (S_{h_j} + Y_j)^2 ]^{1/2} - E[(S_{h_j})^2]^{1/2} | \le E[(Y_j)^2]^{1/2} =\sigma_j.
  \end{align}
  By Lemma \ref{lemma:s_h_j_asymp_sigma_j_B_j}, we have
  \begin{equation}
    s_{h_{j}} \sim s_{h_{j+1}}\quad \text{as $j\to \infty$}.
  \end{equation}
  Since
  \begin{equation}
    s_{h_{j+1}}^2 = E[ ( S_{h_j} + Y_j)^2 ]=s_{h_j}^2 + 2E[S_{h_j}Y_j ] + \sigma_j^2,
  \end{equation}
  we have
  \begin{align}
    |s_{h_{j+1}}^2 -s_{h_j}^2 - \sigma_j^2| &\le 2| E[S_{h_j}Y_j ] |\\
    &\le 2\left|E\left[S_{h_j} \sum_{k=h_j+1}^{h_j+H_j} a_k X_k\right]\right| + 2\left|E\left[S_{h_j} \sum_{k = h_j+H_j+1 }^{h_{j+1}} a_k X_k\right]\right|\\
    &\le 2s_{h_j}(A_{h_j+H_j} -B_j)^{1/2} + 2|\alpha|^{H_{j}/2}s_{h_j} \sigma_{j}\\
    &\ll \sigma_j B_j^{\delta/4} (A_{h_j+H_j} -B_j)^{1/2} + |\alpha|^{H_j/2} B_{j}^{\delta/4} \sigma_j^2 \label{eq:s_h_jplus1_ineq}
  \end{align}
  by the Cauchy--Schwarz inequality, Lemma \ref{lemma:Ibragimov}, Lemma \ref{lemma:momentL2} and Lemma \ref{lemma:s_h_j_asymp_sigma_j_B_j}, where
  \begin{equation}\label{eq:H_n}
  H_n := \bigg\lfloor\frac{24\log B_n}{\log (1/|\alpha|)} \bigg\rfloor +1.
\end{equation}
If $h_j + H_j \ge h_{j+1}$, then we replace $h_j+H_j$ appeared in the first term of \eqref{eq:s_h_jplus1_ineq} with $h_{j+1}$, and the second term of \eqref{eq:s_h_jplus1_ineq} is $0$. By \eqref{eq:max_a_k_upperbound} and \eqref{eq:B_nPlus1_sim_B_n},
\begin{align}
  (A_{h_j+H_j} -B_j)^{1/2} &\ll H_j^{1/2} B_j^{1/2-\delta/2}. 
\end{align}
Noting that
\begin{equation}
  x^{\tfrac{c\log y}{\log (1/x)}} = y^{-c} \quad \text{for any $x\in (0,1)$, $y\in (0,\infty)$ and $c\in (0,\infty)$},
\end{equation}
we have
\begin{equation}\label{eq:alpha_power_H_n}
  |\alpha|^{H_j/2} \ll B_j^{-6}\quad \text{as $j\to \infty$.}
\end{equation}
Thus, we can continue the chain of the inequality \eqref{eq:s_h_jplus1_ineq},
\begin{align}
  |s_{h_{j+1}}^2 -s_{h_j}^2 - \sigma_j^2| &\ll \sigma_j^2 B_{j}^{-\delta/4} H_j^{1/2} + \sigma_j^2 B_{j}^{-6+\delta/4}\\
  &\ll \sigma_j^2 B_{j}^{-\delta/4} (\log B_j)^{1/2} \quad \text{as $j\to \infty$},\label{eq:two_s_h_jplus1_ineq}
\end{align}
which yields \eqref{eq:s_h_j_LLN}. As $M\to \infty$, we obtain
\begin{align}
  \left|s_{h_{M+1}}^2 - \sum_{j=1}^{M} \sigma_{j}^2\right| &\le \sum_{j=1}^{M} |s_{h_{j+1}}^2 -s_{h_j}^2 -\sigma_j^2 |\\
  &\ll \sum_{j=1}^{M} \sigma_j^2 B_{j}^{-\delta/4} (\log B_j)^{1/2}\\
  &\ll (\log B_M)^{1/2} B_M^{1-\delta/8}\sum_{j=1}^{M} \frac{B_{j+1}-B_{j}}{B_{j}^{1+\delta/8}}\\
  &\ll (\log B_M)^{1/2} B_M^{1-\delta/8}
\end{align}
by \eqref{eq:two_s_h_jplus1_ineq}, \eqref{eq:B_nPlus1_sim_B_n}, \eqref{eq:asymptotics_s_h_j_sigma_j} and Theorem \ref{thm:AbelDiniPringsheim}.
\end{proof}

The following lemma is the law of large numbers for $\{Y_j^2\}$.
\begin{lemma}\label{lemma:LLN_Y_j}
  As $M\to \infty$,
  \begin{equation}
    s_{h_{M+1}}^2 - \sum_{j=1}^M Y_j^2 = O(B_M^{1-\delta/16}) \quad \text{a.s.}
  \end{equation}
\end{lemma}
\begin{proof}
  In order to prove Lemma \ref{lemma:LLN_Y_j}, it is sufficient to show
  \begin{claim}\label{claim:Y_j_remainder}
    As $M\to\infty$,
    \begin{equation}
      \sum_{j=1}^{M}(Y_j^2 -\sigma_j^2) \ll B_M^{1-\delta/16} \quad \text{a.s.}
    \end{equation}
  \end{claim}
  \begin{proof}
    Let $Z_j:=Y_j^2-\sigma_j^2$. A direct calculation shows that
    \begin{equation}\label{eq:sum_Z_j_2nd_moment}
      E\biggl[\biggl(\sum_{j=m}^n Z_j\biggr)^2\biggr] = \sum_{j=m}^n E[(Z_j)^2]+ 2\sum_{m\le j<k \le n} E[Z_jZ_k]
    \end{equation}
    for $1\le m<n$.
    By Lemma \ref{lemma:4th_moment},
    \begin{equation}\label{eq:Z_j_2nd_moment}
      E[(Z_j)^2] \le E[(Y_j)^4] \ll (B_{j+1}-B_j)^2.
    \end{equation}
    For the second term of \eqref{eq:sum_Z_j_2nd_moment}, we distinguish the cases $j=k-1$ and $j\le k-2$. We have, by \eqref{eq:Z_j_2nd_moment} and H{\"o}lder's inequality,
    \begin{align}
      E[|Z_jZ_{j+1}|] & \le E[(Z_j)^2]^{1/2} E[(Z_{j+1})^2]^{1/2}\\
      &\ll (B_{j+1}-B_j)(B_{j+2}-B_{j+1}).
    \end{align}
    We consider the case $j\le k-2$. By \eqref{eq:phi_definition}, Lemma \ref{lemma:Ibragimov}, \eqref{eq:Z_j_2nd_moment} and \eqref{eq:alpha_power_H_n},
    \begin{align}
      |E[Z_j Z_k]| &\le 2 (\phi(h_k -h_{j+1}))^{1/2} E[(Z_j)^2]^{1/2} E[(Z_k)^2]^{1/2}\\
      &\ll |\alpha|^{(h_k-h_{j+1})/2} (B_{j+1}-B_j)(B_{k+1}-B_{k})\\
      &\ll |\alpha|^{H_k/2} (B_{j+1}-B_j)(B_{k+1}-B_{k})\\
      &\ll B_{k}^{-6}(B_{j+1}-B_j)(B_{k+1}-B_{k}),
    \end{align}
    where $H_k$ is defined by \eqref{eq:H_n}. Note that, for large $j\le k-2$,
    \begin{equation}
      h_k - h_{j+1} \ge h_k -h_{k-1} \gg H_k,
    \end{equation}
    because, by \eqref{eq:max_a_k_upperbound},
    \begin{equation}
      A_{h_k} - A_{h_{k-1}} \ll (h_k -h_{k-1}) B_k^{1-\delta},
    \end{equation}
    and so
    \begin{equation}
      h_k-h_{k-1} \gg B_{k}^{\delta/2} \gg H_k
    \end{equation}
    by \eqref{eq:definition_h_n} since $\log x \ll x^{\delta/2}$ as $x\to \infty$.
    Substituting these estimates into \eqref{eq:sum_Z_j_2nd_moment}, we find that
    \begin{align}
      E\biggl[\biggl(\sum_{j=m}^n Z_j\biggr)^2\biggr] &\ll \sum_{j=m}^{n} (B_{j+1}-B_j) (B_{j+2}-B_{j})\\
      &\ll \sum_{j=m}^{n} (B_{j+1} -B_{j}) B_{j}^{1-\delta/2}\\
      &\le \sum_{j=m}^{n} \int_{B_{j}}^{B_{j+1}} x^{1-\delta/2} dx\\
      &= \frac{1}{2-\delta/2} (B_{n+1}^{2-\delta/2}-B_{m}^{2-\delta/2}).\label{eq:inequality_sum_Z_j_2nd_moment}
    \end{align}
    Thus, we have, by Chebyshev's inequality and \eqref{eq:inequality_sum_Z_j_2nd_moment},
    \begin{equation}
      P\biggl(\biggl|\sum_{j=1}^m Z_j \biggr| \ge B_m^{1-\delta/8} \biggr) \ll B_m^{-\delta/4}.\label{eq:sum_Z_j_concentration_ineq}
    \end{equation}
    For each positive integer $r$, let $m_r$ be the smallest integer $m$ with $B_m^{\delta/4}\ge r^2$. Then by \eqref{eq:sum_Z_j_concentration_ineq} and the Borel--Cantelli Lemma,
    \begin{equation}\label{eq:m_r_as_convergence}
      \sum_{j=1}^{m_r} Z_j \ll B_{m_r}^{1-\delta/8} \quad \text{a.s.}
    \end{equation}
    Letting $\displaystyle 1<\theta < \frac{2-(3\delta)/8}{2-\delta/2}$, it follows from Lemma \ref{lemma:Billingsley_maximal_ineq} and \eqref{eq:inequality_sum_Z_j_2nd_moment} that
    \begin{align}
      P\biggl(\max_{m_r < m\le m_{r+1}} \biggl|\sum_{j=m_r}^{m} Z_j\biggr| \ge B_{m_r}^{1-\delta/16}\biggr) &\ll B_{m_{r+1}}^{\theta(2-\delta/2)} B_{m_r}^{-2 + \delta/8}\\
      &\ll B_{m_r}^{-\delta/4} \ll r^{-2}.
    \end{align}
    Thus, by the Borel--Cantelli Lemma again, we obtain
    \begin{equation}\label{eq:general_m_convergence}
      \max_{m_r < m\le m_{r+1}} \biggl|\sum_{j=1}^{m} Z_j \biggr| \ll B_{m_r}^{1-\delta/16} \quad \text{a.s.},
    \end{equation}
    which implies Claim \ref{claim:Y_j_remainder}.
  \end{proof}
  Combining \eqref{eq:s_h_Mplus1_sum_sigma_j} and Claim \ref{claim:Y_j_remainder} shows Lemma \ref{lemma:LLN_Y_j}.
\end{proof}

\subsection{The martingale representation}
Let $\mathcal{G}_j$ be the $\sigma$-field generated by $Y_1,\dots,Y_j$. We show that the Gordin type martingale representation works.
\begin{lemma}\label{lema:martingale_representation}
  We have the relation \eqref{eq:Gordin_decomposition}, where $\{\xi_j\}_{j=1}^{\infty}$ is a martingale difference sequence with respect to $\{\mathcal{G}_{j}\}_{j=1}^{\infty}$ and $u_j$ defined by \eqref{eq:u_j_definition} satisfies
  \begin{equation}\label{eq:u_j_upperbound}
    u_j \ll (B_{j+1} - B_j)^{1/2} B_{j}^{-\delta/8} \quad \text{a.s.}
  \end{equation}
\end{lemma}
\begin{proof}
  Let
  \begin{align}
    u_j &= \sum_{\ell = 0}^{\infty} E[Y_{j+\ell} \mid \mathcal{G}_{j-1}]=\sum_{\ell = 0}^{\infty} \sum_{k\in I_{j+\ell}} a_k E[X_k \mid \mathcal{G}_{j-1}]\\
    &= \sum_{k > h_j} a_k \alpha^{k-h_j+1} X_{h_{j-1}} = X_{h_{j-1}} \sum_{k\ge 1} a_{k+h_j} \alpha^{k+1}.\label{eq:u_j}
  \end{align}
  We decompose the sum \eqref{eq:u_j} into two terms: 
  \begin{equation}
    \Sigma_1 := \sum_{k\le H_j}^{} a_{k+h_j} \alpha^{k+1},\quad\text{and}\quad \Sigma_2 := \sum_{k> H_j}^{} a_{k+h_j} \alpha^{k+1}.
  \end{equation}
  By the Cauchy--Schwarz inequality, \eqref{eq:H_n}, \eqref{eq:B_nPlus1_sim_B_n} and \eqref{eq:B_nPlus1_sim_B_n},
  \begin{align}
    |\Sigma_1| &\le \sum_{1\le k\le H_j}^{} |a_{k+h_j}|\cdot |\alpha|^{k+1}  \le \bigg(\sum_{k\le H_j}^{} |a_{k+h_j}|^2 \bigg)^{1/2} \bigg(\sum_{k\le H_j}^{} |\alpha|^{2(k+1)} \bigg)^{1/2} \\
    & \ll H_j^{1/2} B_{j}^{(1-\delta)/2} \ll (B_{j+1} - B_j)^{1/2} B_j^{-\delta/8}.\label{eq:u_j_first}
  \end{align}
  For the second term $\Sigma_2$, we have
  \begin{align}
    |\Sigma_2| &\le \sum_{k> H_j}^{} |a_{k+h_j}|\cdot |\alpha^{k+1}| \ll \sum_{k> H_j}^{} A_{k+h_j}^{(1-\delta)/2} |\alpha|^{k+1}\\
    &\ll A_{h_j}^{(1-\delta)/2}\sum_{k> H_j}^{}  q^{(1-\delta)k/2} |\alpha|^{k+1} \ll B_j^{1/2} \sum_{k>H_j}^{} q^{-(1+\delta)k/2}\\
    &\ll B_j^{1/2} q^{-H_j/2}\ll B_j^{-1}\label{eq:u_j_second}
  \end{align}
  by \eqref{Assump2:last_term}, \eqref{eq:A_j_plus_k_le_A_j_exponential} for $q:=1/|\alpha|$ and \eqref{eq:alpha_power_H_n}. The estimates \eqref{eq:u_j_first} and \eqref{eq:u_j_second} show that the series $u_j$ converges absolutely with probability one. Hence, letting $\xi_j:=Y_j-u_j+u_{j+1}$, we obtain a martingale difference sequence $\{\xi_j,~\mathcal{G}_j\}_{j=1}^{\infty}$.
\end{proof}

\begin{lemma}\label{lemma:mart_diff_LLN}
  As $M\to \infty$, $\displaystyle s_{h_{M+1}}^2 -\sum_{j=1}^{M} \xi_j^2 = O(B_M^{1-\delta/16})$ a.s.
\end{lemma}
\begin{proof}
  Put
  \begin{equation}\label{eq:def_v_j}
    v_j := u_j - u_{j+1}.
  \end{equation}
  By \eqref{eq:u_j_upperbound} and \eqref{eq:B_nPlus1_sim_B_n},
  \begin{equation}
    v_j^2 \ll u_j^2 + u_{j+1}^2 \ll B_{j}^{-\delta/4}(B_{j+2}-B_j)\ll B_{j+2}^{-\delta/4}(B_{j+2}-B_j) \quad \text{a.s.}
  \end{equation}
  It follows from \eqref{eq:B_nPlus1_sim_B_n} and Theorem \ref{thm:AbelDiniPringsheim} that
  \begin{align}
    \sum_{j=1}^{M} v_j^2 &\ll \sum_{j=1}^{M} B_{j+2}^{-\delta/4}(B_{j+2}-B_j)\\
    &\ll B_{M}^{1-\delta/8} \sum_{j=1}^{M} B_{j+2}^{-1-\delta/8}(B_{j+2}-B_j)\ll B_{M}^{1-\delta/8}\quad \text{a.s.}\label{eq:sum_v_j2}
  \end{align}
  Hence, combining Lemma \ref{lemma:LLN_Y_j} and
  \begin{align}
    \sum_{j=1}^{M} (\xi_j^2-Y_j^2) &= \sum_{j=1}^{M} (-2Y_j v_j + v_j^2)\\
    &\ll \bigg(\sum_{j=1}^{M} Y_j^2\bigg)^{1/2} \bigg(\sum_{j=1}^{M} v_j^2\bigg)^{1/2} + \sum_{j=1}^{M} v_j^2\\
    &\ll B_M^{1/2}\cdot B_M^{1/2-\delta/16} + B_M^{1-\delta/8}\\
    &\ll B_{M}^{1-\delta/16} \quad \text{a.s.}
  \end{align}
  yields Lemma \ref{lemma:mart_diff_LLN}.
\end{proof}
\begin{lemma}\label{lemma:mart_conditional_variance}
  As $M\to\infty$, $\displaystyle \sum_{j=1}^{M} \bigl(E[\xi_j^2\mid \mathcal{G}_{j-1}] - \xi_j^2\bigr) \ll B_M^{1-\delta/8}$ a.s.
\end{lemma}
\begin{proof}
  Letting $R_j:= E[\xi_j^2\mid \mathcal{G}_{j-1}] - \xi_j^2$, the sequence $\{R_j,\ \mathcal{G}_j\}_{j=1}^{\infty}$ is a martingale difference sequence. By \eqref{eq:Gordin_decomposition},
  \begin{align}
    E[R_j^2] & \ll E[\xi_j^4] \ll E[Y_j^4] + E[v_j^4]\\
    &\ll (B_{j+1}-B_j)^2 + (B_{j+2}- B_j)^2 B_{j+2}^{-\delta} \ll (B_{j+2} - B_j) B_{j+2}^{1-\delta/2}.
  \end{align}
  By Theorem \ref{thm:AbelDiniPringsheim},
  \begin{equation}
    \sum_{j=1}^{\infty} B_{j}^{-2+\delta/4} E[R_j^2] \ll \sum_{j=1}^{\infty} (B_{j+2} -B_j) B_{j+2}^{-1-\delta/4} <\infty,
  \end{equation}
  and therefore it follows from Doob's convergence theorem (Corollary 2.2 of Hall and Heyde \cite{HallHeydemart}) that
  \begin{equation}
    \sum_{j=1}^{\infty} B_j^{-1+\delta/8} R_j \ll 1 \quad \text{a.s.}
  \end{equation}
  By Kronecker's Lemma, we obtain
  \begin{equation}
    \sum_{j=1}^{M} R_j \ll B_M^{1-\delta/8} \quad \text{a.s.}
  \end{equation}
\end{proof}
\subsection{The embedding argument}
Recall that the standard Brownian motion $\{B(t) \colon t\ge 0\}$ and the sequence $\{T_n\}_{n=1}^{\infty}$ satisfy \eqref{eq:embedding}, \eqref{eq:Skorokhod_tau_n_expectation} and \eqref{eq:Skorokhod_tau_n_square_expectation}.
\begin{lemma}\label{lemma:embedding_time_LLN}
  As $M\to\infty$,
  \begin{equation}
    \sum_{j=1}^{M}\tau_j - s_{h_{M+1}}^2 =  O(B_M^{1-\delta/16})  \quad \text{a.s.}
  \end{equation}
\end{lemma}
\begin{proof} By \eqref{eq:Skorokhod_tau_n_expectation},
  \begin{align}
    \sum_{j=1}^{M} \tau_j - s_{h_{M+1}}^2 &= \sum_{j=1}^{M} (\tau_j - E[\tau_j\mid \mathcal{H}_{j-1}]) + \sum_{j=1}^{M} (E[\xi_j^2 \mid \mathcal{G}_{j-1}] - \xi_j^2)\\
    &\quad + \sum_{j=1}^{M}\xi_j^2 - s_{h_{M+1}}^2.
  \end{align}
  By Lemma \ref{lemma:mart_diff_LLN} and Lemma \ref{lemma:mart_conditional_variance}, we can see that the second term and the third term are $O(B_M^{1-\delta/16})$. Letting $R_j=\tau_j-E[\tau_j\mid \mathcal{H}_{j-1}]$, we obtain a martingale difference $\{R_j, \mathcal{H}_j\}_{j=1}^{\infty}$ with
  \begin{equation}
    E[R_j^2] \ll E[\xi_j^4].
  \end{equation}
  By an argument similar to that for Lemma \ref{lemma:mart_conditional_variance}, we find the first term is also $O(B_M^{1-\delta/16})$.
\end{proof}
If $h_{M}< n \le h_{M+1}$, then
\begin{align}
  s_{h_{M}}^2 - s_{n}^2 &= -2E\biggl[S_{h_{M}} \sum_{k=h_{M}+1}^{n} a_k X_k\biggr] + E\biggl[\biggl(\sum_{k=h_{M}+1}^{n} a_k X_k\biggr)^2\biggr]\\
  &\ll s_{h_{M}} \sigma_{M} + \sigma_{M}^2\\
  &\ll s_{h_{M}}^2 B_{M}^{-\delta/4}
\end{align}
by the Cauchy--Schwarz inequality and the fact
\begin{equation}
  E\biggl[\biggl(\sum_{k=h_{M}+1}^{n} a_k X_k\biggr)^2\biggr] \ll \sigma_M^2.
\end{equation}
Thus, if $h_{M} <n \le h_{M+1}$, then, by Lemma \ref{lemma:embedding_time_LLN},
\begin{equation}\label{eq:sum_tau_j_asymptotic}
  \sum_{j=1}^{M-1} \tau_j  = s_n^{2} + O(B_M^{1-\delta/16})\quad \text{a.s.}
\end{equation}
We show that the remainder $\displaystyle \sum_{k=h_M+1}^{n} a_k X_k$ in the blocking procedure is sufficiently small.
\begin{lemma}\label{lemma:blocking_remainder}
  As $M\to\infty$,
  \begin{equation}\label{eq:blocking_remainder}
    \max_{h_M<n\le h_{M+1}}\left|\sum_{k=h_M+1}^{n} a_k X_k\right| \ll B_M^{1/2-\delta/16} \quad \text{a.s.}
  \end{equation}
\end{lemma}
\begin{proof}
  For any $n\in (h_{M}, h_{M+1}]$, by Markov's inequality, Lemma \ref{lemma:4th_moment} and \eqref{eq:B_nPlus1_sim_B_n},
  \begin{align}
    P\biggl(\biggl|\sum_{k=h_M+1}^{n} a_k X_k\biggr| \ge B_M^{1/2-\delta/16}\biggr) &\le K_2\cdot B_{M}^{-2-\delta/4} (B_{M+1}-B_{M})^2\\
    & \le K_2 B_{M}^{-1-\delta/4} (B_{M+1} - B_{M}).
  \end{align}
  By Lemma \ref{lemma:Billingsley_maximal_ineq}, Theorem \ref{thm:AbelDiniPringsheim} and the Borel--Cantelli Lemma, we have \eqref{eq:blocking_remainder}.
\end{proof}
Finally, let $n\ge 1$ be any integer. Choosing $M$ such that $n\in (h_{M}, h_{M+1}]$, we have
\begin{align}
  \sum_{k=1}^{n} a_k X_k - B(s_n^2) &= \sum_{j=1}^{M-1} Y_j + \sum_{j=h_{M}+1}^{n} a_k X_k - B(s_n^2)\\
  &= \sum_{j=1}^{M-1} (u_j-u_{j+1}) + \sum_{j=h_{M}+1}^{n} a_k X_k + B\biggl(\sum_{j=1}^{M-1} \tau_j\biggr) - B(s_n^2) \\
  &\ll u_{M} + \max_{h_{M}<n \le h_{M+1}} \biggl|\sum_{j=h_{M}+1}^{n} a_k X_k\biggr|\\
  &\quad + \sup_{h_{M}<n\le h_{M+1}}\biggl|B\biggl(\sum_{j=1}^{M-1} \tau_j\biggr) - B(s_n^2)\biggr|\\
  &\ll B_M^{1/2-\lambda_1} \ll s_{h_{M}}^{1-2\lambda_1} \ll s_n^{1-2\lambda_1} \quad \text{a.s.}
\end{align}
for each $\lambda_1 <\delta/32$ by \eqref{eq:Gordin_decomposition}, \eqref{eq:embedding}, \eqref{eq:u_j_upperbound}, Lemma \ref{lemma:blocking_remainder}, \eqref{eq:sum_tau_j_asymptotic} and Hanson and Russo \cite[Theorem 3.2A]{HansonRusso83}.

\section*{Acknowledgements}
I am very grateful to my supervisor Masato Takei for his helpful comments on this paper. This work was supported by JST SPRING, Japan Grant Number JPMJSP2178.

\appendix
\section{Some technical results}\label{appendix:some_technical_results}
We frequently use the following Abel--Dini--Pringsheim Theorem.
\begin{theorem}[see e.g., Knopp \cite{Knopp56InfiniteSequences}, pp.~ 125--126]\label{thm:AbelDiniPringsheim}
    Let $\{d_n\}_{n\ge 1}$ be a positive sequence with $\sum_{n=1}^{\infty} d_n = +\infty$, and $D_n=d_1+\dots +d_n$ be its partial sums. Then
    \begin{equation}
      \sum_{n=1}^{\infty} \frac{d_n}{D_n}=+\infty,\quad \text{and} \quad\sum_{n=1}^{\infty} \frac{d_n}{D_n D_{n-1}^{\varepsilon}} <+\infty
    \end{equation}
    for any $\varepsilon > 0$.
  \end{theorem}
We  have an upper bound for the fourth moment of $S_n$ defined by \eqref{eq:Def_ERWVRP_with_variable_step_length}.
\begin{lemma}[{Yoshihara \cite[Theorem 1]{Yoshihara78moment}}]\label{lemma:4th_moment}
  There exists a positive constant $K_2>0$, such that, for any $\{a_k\}_{k\ge 1}$ and any integers $n>m\ge 1$,
  \begin{equation}
    E[(S_n -S_m)^4] \le K_2 (A_n -A_m)^2.
  \end{equation}
\end{lemma}
We use the following covariance inequalities (Ibragimov's inequalities) for general $\phi$-mixing sequences.
\begin{lemma}[Ibragimov \cite{Ibragimov62}]\label{lemma:Ibragimov}
  Let $U$ be $\sigma\{X_i\colon 1\le i\le n\}$-measurable and $V$ be $\sigma\{X_i\colon n+m\le i\}$-measurable for some integers $n,m\ge 1$. Let $r^{-1}+s^{-1}=1$ with $r> 1$. Suppose that $E[|U|^r]<\infty$ and $E[|V|^s]<\infty$. Then,
  \begin{equation}
    |E[UV]-E[U]E[V]| \le 2 (E[|U|^r])^{1/r} (E[|V|^s])^{1/s} \phi(m)^{1/r}.
  \end{equation}
\end{lemma}
\begin{lemma}[Ibragimov \cite{Ibragimov62}]\label{lemma:Ibragimov2}
  Let $U$ be $\sigma\{X_i\colon 1\le i\le n\}$-measurable and $V$ be $\sigma\{X_i\colon n+m\le i\}$-measurable for some integers $n,m\ge 1$. Suppose that $|U|\le 1$ a.s. and $|V|\le 1$ a.s. Then,
  \begin{equation}
    |E[UV]-E[U]E[V]| \le 4 \phi(m).
  \end{equation}
\end{lemma}
Billingsley's maximal inequality is also useful.
\begin{lemma}[{Billingsley \cite[Theorem 10.2]{Billingsley}}]\label{lemma:Billingsley_maximal_ineq}
  Let $\{Z_k\}_{k=1}^{\infty}$ be a sequence of random variables and let $\gamma\ge 0$ and $\theta >1$. Suppose that there exists a sequence $\{c_k\}_{k=1}^{\infty}$ of non-negative numbers such that
  \begin{equation}
    P\biggl(\biggl|\sum_{k=m}^{n} Z_k\biggr| \ge  \lambda\biggr) \le \lambda^{-\gamma} \biggl(\sum_{k=m}^{n} c_k\biggr)^{\theta}
  \end{equation}
  holds for all positive $\lambda$ and all integers $1\le m < n\le N$. Then, for all positive $\lambda$,
  \begin{equation}
    P\biggl(\max_{m < \ell \le N}\biggl|\sum_{k=m}^{\ell} Z_k\biggr| \ge  \lambda\biggr) \ll \lambda^{-\gamma} \biggl(\sum_{k=m}^{N} c_k\biggr)^{\theta}.
  \end{equation}
  Here $\ll$ depends only on $\gamma$ and $\theta$.
\end{lemma}

   The following lemma is the Skorokhod embedding theorem for martingales.
  \begin{lemma}[{see e.g., Hall and Heyde \cite[Theorem A.1]{HallHeydemart}}]\label{lemma:Skorokhodembeddingtheorem}
    Let $\{M_n\}_{n=1}^{\infty}$ be a zero-mean and square-integrable martingale with respect to the $\sigma$-field $\mathcal{F}_n$ generated by $M_1,\dots,M_n$. Then there exists a standard Brownian motion $\{B(t) \colon t\ge 0\}$ and a non-decreasing sequence $\{T_n\}_{n=1}^{\infty}$ of random variables such that
    \begin{equation}\label{eq:embedding}
      M_n = B (T_n) \quad \text{a.s.}\quad \text{for all $n\ge 1$},
    \end{equation}
    and there exists a positive constant $C_1 > 0$ such that
    \begin{align}
      E[\tau_n\mid \mathcal{G}_{n-1}] &= E[(\Delta M_n)^2 \mid \mathcal{F}_{n-1}] \quad \text{a.s., and} \label{eq:Skorokhod_tau_n_expectation}\\
      E[ (\tau_n)^2 \mid \mathcal{G}_{n-1}] &\le C_1 E[(\Delta M_n)^4 \mid \mathcal{F}_{n-1}]\quad \text{a.s.,}\label{eq:Skorokhod_tau_n_square_expectation}
    \end{align}
    where $M_0:= 0$, $T_0:=0$, $\tau_n := T_n-T_{n-1}$, $\Delta M_n := M_n-M_{n-1}$, and $\mathcal{G}_{n}$ is the $\sigma$-field generated by $\Delta M_1,\dots, \Delta M_n$ and by $ \{B(t) \colon 0\le t\le T_n \}$.
  \end{lemma}

  \section{Limsup in \eqref{eq:LIL_scaled_A_n}}\label{appendix:limsup_scaled_A_n}
  Fix $p\in (0,1)$. We show that
  \begin{equation}
    \limsup_{n\to \infty} \frac{S_n }{\sqrt{2A_n \log\log A_n} } = \sqrt{\frac{2p^2-2p+1}{2p(1-p)}} \quad \text{a.s.}\label{eq:LIL_scaled_A_n_intermediate}
  \end{equation}
  holds for
  \begin{equation}
    a_k :=
    \begin{dcases*}
      1 & for odd $k$,\\
      0 & for even $k$.
    \end{dcases*}
  \end{equation}
  We obtain
  \begin{align}
    s_n^2 &= \sum_{k=1}^{n} a_k^2 + 2\sum_{1\le k < l \le n}^{} a_k a_l (2p-1)^{l-k}\\
    &=\ceil{\frac{n}{2}} + 2 \sum_{i=1}^{ \ceil{\frac{n}{2}} -1 } \biggl(\ceil{\frac{n}{2}} - i \biggr) (2p-1)^{2i} \\
    &\sim \ceil{\frac{n}{2}} + 2 \ceil{\frac{n}{2}} \sum_{i=1}^{\infty} (2p-1)^{2i} = \frac{2p^2-2p+1}{2p(1-p)} \cdot \ceil{\frac{n}{2}}\quad \text{as $n\to\infty$},
  \end{align}
  which implies the LIL \eqref{eq:LIL_scaled_A_n_intermediate}.

\end{document}